\documentclass[12pt]{article}
\usepackage{amsmath,amssymb,amsfonts,amsthm}
\usepackage{eucal,oldgerm}
\usepackage[dvips]{graphics}
\usepackage{pstricks,pst-plot,pst-node,pst-coil}
\newpsobject{showgrid}{psgrid}{subgriddiv=1,griddots=5,gridlabels=0pt}
\usepackage[arrow,matrix,tips,2cell]{xy}

\newcommand{\lang}{\left\langle}
\newcommand{\rang}{\right\rangle}

\newcommand{\zz}{{\mathfrak{z}}}

\newcommand{\com}{{\mathbb C}}

\newcommand{\bC}{\mathsf{C}}
\newcommand{\CR}{\mathsf{CR}}

\newcommand{\cV}{\mathcal{V}}

\newcommand{\cF}{\mathcal{F}}

\newcommand{\bE}{\mathsf{E}}

\newcommand{\bd}{{\boldsymbol{\delta}}}

\newcommand{\fI}{\textfrak{I}}

\newcommand{\MM}{\mathsf{M}}

\newcommand{\C}{\mathbb{C}}
\newcommand{\Q}{\mathbb{Q}}
\newcommand{\Z}{\mathbb{Z}}

\newcommand{\cO}{\mathcal{O}}
\newcommand{\bO}{\mathsf{O}}

\newcommand{\Pp}{{\mathbf{P}^1}}

\newcommand{\cI}{\mathcal{I}}

\newcommand{\rarr}{\rightarrow}
\newcommand{\ch}{{\text{ch}}}
\newcommand{\ft}{\mathfrak{t}}

\newcommand{\bR}{\mathsf{R}}
\newcommand{\bA}{\mathcal{A}}
\newcommand{\bZ}{\mathsf{Z}}
\newcommand{\oM}{\overline{M}}

\newcommand{\eqq}{\stackrel{\sim}{=}}

\DeclareMathOperator{\Hilb}{Hilb}

\DeclareMathOperator{\Chow}{Chow}
\DeclareMathOperator{\Aut}{Aut}

\DeclareMathOperator{\Lie}{Lie}

\DeclareMathOperator{\Ext}{Ext}

\DeclareMathOperator{\vir}{vir}

\newtheorem{Theorem}{Theorem}
\newtheorem{Lemma}{Lemma}
\newtheorem{Corollary}{Corollary}

\newtheorem{Proposition}[Lemma]{Proposition}
\newtheorem{Conjecture}{Conjecture}

\begin{document}
\title{Gromov-Witten/Donaldson-Thomas correspondence for toric
3-folds}
\author{D.~Maulik, A.~Oblomkov, A.~Okounkov, and R.~Pandharipande}
\date{September 2008}
\maketitle

\begin{abstract}
We prove the equivariant Gromov-Witten theory of a
nonsingular toric 3-fold $X$ with primary insertions 
is equivalent to the equivariant Donaldson-Thomas theory 
of $X$. As a corollary, the topological vertex calculations 
by Agangic, Klemm, Mari\~no, and Vafa of
the Gromov-Witten theory of local Calabi-Yau toric $3$-folds
are proven to be correct in the full 3-leg setting.
\end{abstract}

\setcounter{tocdepth}{2}
\tableofcontents

\section{The GW/DT correspondence}
\label{pppr}
\subsection{Gromov-Witten theory}
\label{gwt}
Let $X$ be a nonsingular projective 3-fold. Let
 $\oM'_{g,r}(X,\beta)$ denote the moduli space of
$r$-pointed stable maps
$$
f: C \to X
$$
from possibly disconnected
genus $g$ curves to $X$ 
representing the
class $$\beta\in H_2(X,\Z)$$
and {\em not} collapsing any connected components.

Denote the evaluation map corresponding to the $i^{th}$
marked point by
$$
\text{ev}_i: \oM'_{g,r}(X,\beta) \rarr X.
$$
Given classes  $\gamma_i\in H^*(X,\mathbb{Z})$,
the corresponding \emph{primary} Gromov-Witten invariants are
defined by
\begin{equation}
\lang \gamma_1,\dots,\gamma_r \rang'_{g,\beta} =
\int_{\left[\oM_{g,r}'(X,\beta)\right]^{\vir}}
\prod_{i=1}^r \text{ev}_i^*(\gamma_{i}) \,,
\label{prgw}
\end{equation}
the virtual counts of genus $g$ degree
 $\beta$  curves
in $X$ meeting cycles Poincar\'e dual to $\gamma_i$.
The integration in \eqref{prgw} is against the \emph{virtual
fundamental class} of dimension
$$
\dim \left[\overline{M}'_{g,r}(X,\beta)\right]^{\vir}
= \int_\beta c_1(T_X) + r \,.
$$
We will often call the primary fields $\gamma_i$ {\em insertions}.
Foundational aspects of the theory
are treated, for example, in \cite{Beh, BehFan, LiTian}.

We assemble the primary invariants into the 
generating function
\begin{equation}
\label{abc}
\bZ'_{GW}\left(\left. X,u\ \right|
\gamma_1,\dots,\gamma_r\right)_\beta =
\sum_{g\in\Z} \lang \gamma_1,\dots,\gamma_r \rang'_{g,\beta} \, u^{2g-2}.
\end{equation}
Since the domain components must map nontrivially, an elementary
argument shows the genus $g$ in the  sum \eqref{abc} is bounded from below.
Following the terminology of \cite{mnop1},
we view \eqref{abc} as a  \emph{reduced} partition function.

For brevity, 
we will often omit the arguments of $\bZ'_{GW}$ which 
are clear from context. The
genus subscript, however,
 will follow a different convention.
When  the genus is not present, a summation over
all genera is understood
$$
\lang \gamma_1,\dots,\gamma_r \rang'_{\beta} =
\bZ'_{GW}\left(\left. X,u\ \right|
\gamma_1,\dots,\gamma_r\right)_\beta \,.
$$

\subsection{{Donaldson-Thomas theory}}
\label{uch}
Donaldson-Thomas theory is defined via integration over
the moduli space of ideal sheaves \cite{DT,T} of $X$.
An {\em ideal sheaf} is a torsion-free sheaf of rank 1 with
trivial determinant. For an  ideal sheaf ${\mathcal I}$, the
canonical map
$$0 \rarr {\mathcal I} \rarr {\mathcal I} ^{\vee\vee}$$
is an injection.
As ${\mathcal I}^{\vee\vee}$ is reflexive of rank 1 with trivial
determinant,
$${\mathcal I}^{\vee\vee}\eqq \cO_X,$$
see \cite{OSS}.
Each ideal sheaf $\cI$
determines a subscheme $Y\subset X$,
$$
0 \rarr {\mathcal I} \rarr \cO_X \rarr \cO_{Y} \rarr 0.
$$
By the triviality of the determinant,
the components of the subscheme $Y$
have dimension at most $1$.
The 1-dimensional components of $Y$
(weighted by their intrinsic multiplicities) determine
an element,
$$[Y] \in H_{2}(X,\Z).$$
Let $I_n(X,\beta)$ denote the moduli space of ideal sheaves
${\mathcal I}$ satisfying
$$\chi(\cO_{Y}) = n \ \ \text{and} \ \ 
[Y] = \beta\in H_2(X,\Z).$$
Here, $\chi$ denotes the holomorphic Euler characteristic.

The Donaldson-Thomas invariant is defined
via integration against virtual class
$[I_n(X,\beta)]^{\vir}$ of dimension
$$\text{dim} [I_n(X,\beta)]^{\vir} = \int_\beta c_1(T_X).$$
Foundational aspects of the theory
are treated in \cite{mnop1,T}.

The moduli space $I_n(X,\beta)$ is canonically isomorphic to the
Hilbert scheme \cite{mnop1}. As the Hilbert scheme is a fine moduli
space, we have the \emph{universal ideal sheaf}
\begin{equation}
\fI \rarr I_n(X,\beta) \times X\,.
\label{uIs}
\end{equation}
Let $\pi_1$ and $\pi_2$ denote the projections to the respective
factors of $I_n(X,\beta) \times X$.
Since $\fI$ is $\pi_1$-flat and $X$ is nonsingular,
a finite resolution
of $\fI$ by locally free sheaves on $I_n(X,\beta) \times X$
exists. Hence, the Chern classes of $\fI$ are well-defined.
The second Chern class $c_2(\fI)$ may be interpreted as
the class of the universal subscheme.

For each class $\gamma\in H^*(X, \mathbb{Z})$,
let $c_{2}(\gamma)$ denote the 
operator{\footnote{Since $c_1(\fI)$
vanishes, $c_2(\fI)=-\text{ch}_2(\fI)$. Thus, the
definition of primary field insertions here agrees with the
definition of \cite{mnop2}.}}
on the homology
$H_*(I_n(X,\beta), \mathbb{Z})$
defined by
\begin{equation}
c_{2}(\gamma)\big(\xi\big)=
 \pi_{1*}\big(   c_{2}(\fI)
\cdot \pi_2^*(\gamma) \cap      \pi_1^*(\xi)\big)\,.
\label{c2oper}
\end{equation}
Since $\pi_1$ is flat, the homological pull-back $\pi_1^*$ is well-defined
\cite{xxx}.

The primary Donaldson-Thomas
invariants are defined by
\begin{equation}
\lang \gamma_1,\dots,\gamma_r \rang_{n,\beta} =
\int_{[{I}_n(X,\beta)]^{\vir}}
\prod_{i=1}^r c_{2}(\gamma_{i})\,,
\label{prdt}
\end{equation}
where the latter integral is the push-forward to a point of
the class
$$
c_{2}(\gamma_{1}) \ \circ\ \cdots\
\circ\  c_{2}(\gamma_{r})\left(
\left[I_n(X,\beta)\right]^{\vir}\right).
$$
A similar slant product construction can be found in the Donaldson
theory of 4-manifolds.

As in Gromov-Witten theory, we assemble primary invariants into a
generating function,
\begin{equation}
\label{abb}
\bZ_{DT}\left(\left. X,q\ \right|
\gamma_1,\dots,\gamma_r\right)_\beta =
\sum_{n\in\Z} \lang \gamma_1,\dots,\gamma_r \rang_{n,\beta} \, q^n.
\end{equation}
An elementary
argument shows $n$ in \eqref{abb} is bounded from below.

The reduced  partition function
is obtained by formally removing the
degree 0 contributions,
$$
\bZ'_{DT}\left(\left. X,q\ \right|
\gamma_1,\dots,\gamma_r\right)_\beta =
\frac{\bZ_{DT}\left(\left. X,q\ \right|
\gamma_1,\dots,\gamma_r\right)_\beta}
{\bZ_{DT}(X,q)_0}\,.
$$
The evaluation of  the
 degree 0 partition function was conjectured in  \cite{mnop1,mnop2}
and proven in
 \cite{LP,li},
$$
\bZ_{DT}(X,q)_0= M(-q)^{\int_X c_3(T_X \otimes K_X)}\,,
$$
where
\begin{equation}\label{macf}
M(q) = \prod_{n\geq 1} \frac{1}{(1-q^n)^n}
\end{equation}
is the McMahon function.

\subsection{Primary GW/DT correspondence}
\label{pgwdt}
The primary correspondence consists of two claims \cite{mnop1,mnop2}:
\begin{enumerate}
\item[(i)]
The series
$\bZ'_{DT}$
is a rational function of $q$.
\item[(ii)] After the change of variables $e^{iu}=-q$,
$$
(-iu)^\bd\,  \bZ'_{GW}
\left(\left. X,u\ \right|
\gamma_1,\dots,\gamma_r\right)_\beta
=
(-q)^{-\bd/2}\,
\bZ'_{DT}
\left(\left. X,q\ \right|
\gamma_1,\dots,\gamma_r\right)_\beta
\,,
$$
where $\bd= \int_\beta c_1(T_X)$ is the virtual dimension.
\end{enumerate}

The main result of the paper is the proof of
the primary correspondence for toric $X$. Let $T$
be the 3-dimensional torus acting on $X$.

\begin{Theorem} \label{aaa}
The primary GW/DT correspondence holds for all nonsingular
toric 3-folds $X$ in $T$-equivariant cohomology.
\end{Theorem}

In the toric case, both $Z'_{GW}$ and $Z'_{DT}$ take values in
the $T$-equivariant cohomology ring of a point
$$
H^*_T(\bullet,\mathbb{Q}) = {\text{Sym}}( \ft^*)\,, \quad \ft=\Lie T \,.
$$
The total codimension of the
insertions in the $T$-equivariant theory
is allowed to exceed the virtual dimension.
Since both theories are well-defined by residues in the
noncompact case, we do not require $X$ to be projective
for Theorem \ref{aaa}. 

If $X$ is projective, the usual
Gromov-Witten and Donaldson-Thomas invariants are
obtained in the non-equivariant limit. Hence,
Theorem 1 implies the usual GW/DT correspondence in the
nonsingular projective toric case.

\subsection{Relative GW/DT correspondence}
In addition to the primary correspondence, the
correspondence of \emph{relative} theories 
will play an important role.

Let $D\subset X$ be a smooth divisor. Relative Gromov-Witten and 
Donaldson-Thomas
theories enumerate curves with specified tangency to
the divisor $D$. See \cite{mnop2} for a technical discussion
of relative theories.

In Gromov-Witten theory, relative conditions  are
represented by a partition $\mu$ of the number
$$
d=\beta\cdot[D],
$$
each part $\mu_i$ of which is marked
by a cohomology class $\gamma_i\in H^*(D,\mathbb{Z})$. 
The numbers $\mu_i$ record
the multiplicities of intersection with $D$
while the cohomology labels $\gamma_i$ record where the tangency occurs.
More precisely, we integrate
the pull-backs of $\gamma_i$
via the evaluation maps
$$
\oM_{g,r}'(X/D,\beta) \to D
$$
at the points of tangency.
By convention, an absent cohomology label stands for
$1\in H^*(D,\mathbb{Z})$.

In Donaldson-Thomas theory, the relative moduli space admits a natural
morphism to the Hilbert scheme of
$d$ points in $D$. Cohomology classes on $\Hilb(D,d)$ may thus
be pulled back to the relative moduli space. We will work in
the \emph{Nakajima basis} of $H^*(\Hilb(D,d),\mathbb{Q})$ indexed
by a partition $\mu$ of $d$
labeled by cohomology classes of $D$. For example, the
class
$$
\left.\big|\mu\rang \in H^*(\Hilb(D,d),\mathbb{Q})\,,
$$
with all cohomology labels equal to the identity,
 is $\prod \mu_i^{-1}$ times
the Poincar\'e dual of the closure of the subvariety formed by unions of
schemes of length
$$
\mu_1,\dots, \mu_{\ell(\mu)}
$$
supported at $\ell(\mu)$ distinct points of $D$.

The {\em Fock space} associated to $D$ is the sum
$$\mathcal{F} = \bigoplus_{d\geq 0} H^*(\Hilb(D,d),\mathbb{Q}).$$
The Fock space is equipped with the Nakajima basis, a
natural inner product from the Poincar\'e pairing on
$H^*(\Hilb(D,d),\mathbb{Q})$, and standard{\footnote{We will
follow the Fock space notation of \cite{GWDT}.}} operators $\alpha_{\pm r}$
for raising and lowering.

The conjectural relative GW/DT correspondence \cite{mnop2}
equates the generating functions
\begin{multline}\label{relcor}
(-iu)^{\bd+\ell(\mu)-|\mu|}\,  \bZ'_{GW}
\left(\left. X/D,u\ \right|
\gamma_1,\dots,\gamma_r \, \big| \mu  \right)_\beta
\overset{\textup{?}}=\\
(-q)^{-\bd/2}\,
\bZ'_{DT}
\left(\left. X/D,q\ \right|
\gamma_1,\dots,\gamma_r\,\big| \mu \right)_\beta
\,,
\end{multline}
after the change of variables $e^{iu}=-q$. Here,
$\bd= \int_\beta c_1(T_X)$ is the virtual dimension, and
$\mu$ is a cohomology weighted partition with
$\ell(\mu)$ parts. As before, \eqref{relcor} is
conjectured to be a rational function of $q$.

\subsection{Degeneration formulas}\label{sdeg}

Relative theories satisfy degeneration formulas. Let
$$
\textfrak{X}\to B
$$
be a nonsingular $4$-fold fibered over an irreducible and
nonsingular base curve $B$. Let $X$ be a nonsingular
fiber and
$$
X_1 \cup_{D} X_2
$$
be a reducible special fiber consisting of two nonsingular
$3$-folds intersecting transversally along a nonsingular
surface $D$.

If all insertions $\gamma_1,\dots,\gamma_r$ lie in the
image of
$$
H^*(X_1 \cup_{D} X_2,\mathbb{Z}) \to H^*(X,\mathbb{Z})\,,
$$
the degeneration formula in Gromov-Witten theory takes the form
\cite{IP, LR,L}
\begin{multline}\label{degGW}
 \bZ'_{GW}
\left(\left. X\right|
\gamma_1,\dots,\gamma_r\, \right)_\beta =\\
\sum \bZ'_{GW}
\left(\left. X_1\right|
\,\dots\, \big| \mu  \right)_{\beta_1}  \,
\zz(\mu) \, u^{2\ell(\mu)} \,
\bZ'_{GW}
\left(\left. X_2\right|
\,\dots\, \big| \mu^\vee  \right)_{\beta_2} \,,
\end{multline}
where the summation is over all curve splittings
$\beta=\beta_1+\beta_2$, all splitting of the
insertions $\gamma_i$, and all relative conditions $\mu$.

In \eqref{degGW}, the cohomological labels of
$\mu^\vee$ are Poincar\'e duals of the labels of $\mu$.
The gluing factor $\zz(\mu)$ is the order of the
centralizer of in the symmetric group $S(|\mu|)$ of
an element with cycle type $\mu$.

The degeneration formula in Donaldson-Thomas theory takes a very similar
form, 
\begin{multline*}
 \bZ'_{DT}
\left(\left. X\right|
\gamma_1,\dots,\gamma_r\, \right)_\beta =\\
\sum \bZ'_{DT}
\left(\left. X_1\right|
\,\dots\, \big| \mu  \right)_{\beta_1}  \, 
(-1)^{|\mu|-\ell(\mu)} \,
\zz(\mu) \, q^{-|\mu|} \,
\bZ'_{DT}
\left(\left. X_2\right|
\,\dots\, \big| \mu^\vee  \right)_{\beta_2} \,,
\end{multline*}
see \cite{mnop2}. The sum over the relative conditions
$\mu$ is interpreted as the coproduct of $1$,
$$
\Delta 1 = \sum_{\mu}
(-1)^{|\mu|-\ell(\mu)} \, \zz(\mu) \, 
\left.\big|\mu\rang \otimes \left.\big|\mu^\vee\rang\ ,
$$
in the
tensor square of $H^*(\Hilb(D,\beta\cdot [D]),\mathbb{Z})$.
Conjecture \eqref{relcor} is easily seen to be compatible with
degeneration.

\subsection{$\bA_n$ geometries} \label{angeo}
Let $\zeta$ be a primitive $(n+1)^{th}$ root
of unity, for $ n \geq 0$.
Let the generator of the
cyclic group $\mathbb{Z}_{n+1}$ act on $\com^2$ by
$$
(z_1,z_2)\mapsto  (\zeta\, z_1, \zeta^{-1}z_2)\,.
$$
Let  $\bA_n$ be  the minimal resolution of the quotient
$$
\bA_n \rightarrow \com^2/\mathbb{Z}_{n+1}.
$$
The diagonal $(\C^*)^2$-action on $\com^2$ commutes
with the action of $\mathbb{Z}_n$. As a result, the
surfaces $\bA_n$ are toric.

The starting point of our paper is the $(\C^*)^2$-equivariant
GW/DT correspondence for the 3-folds $\bA_n\times \Pp$.
In fact, there
exists a conjectural
triangle of equivalences. 
The apex in that triangle is given by the $(\C^*)^2$-equivariant
quantum cohomology of the Hilbert scheme of points of $\bA_n$.

\begin{figure}[hbtp]\psset{unit=0.5 cm}
  \begin{center}
    \begin{pspicture}(-6,-2)(10,6)
    \psline(0,0)(2,3.464)(4,0)(0,0)
    \rput[rt](0,0){
        \begin{minipage}[t]{3.64 cm}
          \begin{center}
            Gromov-Witten \\ theory of $\bA_n  \times \Pp$
          \end{center}
        \end{minipage}}
    \rput[lt](4,0){
        \begin{minipage}[t]{3.64 cm}
          \begin{center}
             Donaldson-Thomas\\ theory of $\bA_n \times \Pp$
          \end{center}
        \end{minipage}}
    \rput[cb](2,4.7){
        \begin{minipage}[t]{4 cm}
          \begin{center}
           Quantum cohomology \\ of $\Hilb(\bA_n)$
          \end{center}
        \end{minipage}}
    \end{pspicture}
  \end{center}
\end{figure}

The triangle is established for $\bA_0 = \C^{2}$
in the series of papers \cite{BP,OP6,GWDT}.
For the $\bA_{n\geq 1}$ geometries,
the triangle is 
proven{\footnote{In fact, assuming a conjectural
 nondegeneracy statement for the divisorial operators, 
the full equivalence of
the triangle follows, see \cite{mo1}.}}
for a special class of operators
associated to divisors of $\bA_n$ in \cite{dmt, mo1, mo2}.
The equivalence for divisorial operators
 is sufficient for our purposes.
The Hilbert scheme vertex plays a crucial role in
the establishment of the triangle.

\subsection{Plan of the paper}

We prove the GW/DT correspondence 
for  toric $3$-folds $X$ by the following method.
We  decompose the invariants on
both sides in terms of vertex and edge contributions of 
the polytope associated to $X$. Such a decomposition
is obtained from the localization formulas for the respective
virtual classes
\cite{GraberP}. A crucial step is to reorganize the
localization data into {\em capped}  vertex and edge contributions
discussed in Section 2.
The capped contributions are proven to be much better behaved ---
the capped contributions themselves satisfy the GW/DT correspondence.
The Gromov-Witten and Donaldson-Thomas
capped
edges are matched in Section 3.3
by strengthening the local curves correspondence
of \cite{BP,GWDT}. The capped vertices are matched
in Sections 3.4-3.6
 by exploiting the $\bA_n\times \Pp$
geometries for $n\leq 2$ and using the established
properties of the triangle of equivalences for divisorial operators.

Several results, conjectures, and speculations about
 the capped vertex are discussed in Section 4.
The connection of our work to the topological vertex of \cite{AKMV} 
is explained in Section 4.1. Theorem 1 is recast symmetrically
in terms
of classes on the Chow variety of $X$ in Section 4.2.
Finally, examples of capped vertex evaluations are
given in Section 4.3.

\subsection{Acknowledgments}
We thank J. Bryan,  J. Li, M. Lieblich, C.-C. Liu, 
N. Nekrasov, and R. Thomas for many discussions
about the correspondence for toric varieties.

D.~M. was partially supported by an NSF graduate fellowship and
the Clay foundation. 
A.~Ob. was partially supported by the NSF.
A.~Ok. and R.~P. were partially supported
by the Packard foundation and the NSF.

\section{Localization}
\subsection{Toric geometry}
Let $X$ be a nonsingular toric $3$-fold.
For both Gromov-Witten and Donaldson-Thomas theory, 
virtual localization with respect to the
action of the full 3-dimensional torus $T$ reduces all invariants
of $X$
to local contributions of the vertices and edges of the 
associated toric polytope. 
However, the standard
constituent pieces of the localization formula
are poorly behaved with respect to the GW/DT correspondence.  
Instead,
we will introduce a modified localization procedure 
with capped versions of the usual vertex and edge contributions.  
The capped vertex and edge terms are still 
building blocks for global toric calculations, 
but have the advantage of satisfying the GW/DT correspondence
for residues.

Let $\Delta$ denote the polytope associated to
$X$. The vertices of $\Delta$ are in bijection
with $T$-fixed points $X^T$.
The edges $e$ correspond
to $T$-invariant curves $$C_e\subset X.$$ The three edges
incident to any vertex carry canonical $T$-weights ---
the tangent weights of the torus action.

We will consider both compact and noncompact toric
varieties $X$. In the latter case, edges 
maybe compact or noncompact.
Every compact edge is incident to two vertices.

\subsection{Standard localization}

\subsubsection{Structure}

The large-scale structure of the $T$-equivariant localization
formulas in Gromov-Witten and
Donaldson-Thomas theories \cite{GraberP} is nearly identical.
The outermost sum in the localization formula
runs over all assignments of partitions $\lambda(e)$ to
the compact edges $e$ of $\Delta$ satisfying
$$
\beta = \sum_e |\lambda(e)|\cdot \left[C_e\right] \in H_2(X,\mathbb{Z})\,.
$$
Such a partition assignment will be called a \emph{marking} of
$\Delta$.
The weight of each marking in the localization sum for the
invariants \eqref{prgw} and \eqref{prdt} equals the product
of three factors:
\begin{enumerate}
\item[(i)] localization of the integrand,
\item[(ii)] vertex contributions,
\item[(iii)] compact edge contributions.
\end{enumerate}

The complexity of the localization formula is concentrated
in the vertex contributions (ii). The principal ingredients of
the GW vertex are \emph{triple Hodge integrals}.{\footnote{An introduction
to Hodge integrals in Gromov-Witten theory can be found in
\cite{FP}. See \cite{AKMV} for the triple Hodge integrals
which occur in the GW vertex in the Calabi-Yau case.}} 
The DT vertex is a weighted sum over all 3-dimensional
partitions $\pi$ with fixed asymptotic cross-sections. 
The weight of $\pi$ is a rational function of the
equivariant parameters, see \cite{mnop1,mnop2}.
For the computation of the reduced invariants, the \emph{reduced
vertex}, in which the appropriate power of the McMahon
function has been taken out, is more convenient.

We will avoid the fine structure of the GW and DT vertices. 
A precise match will be proven between
\emph{capped} vertices in the two theories.
The capped vertices are $T$-equivariant residue invariants of 
open 1-vertex geometries defined in Section \ref{sCapV}.

\subsubsection{Gromov-Witten edges}
Let $e$ be an compact edge of $\Delta$ as
in Figure \ref{f_edge}. The edge labels in Figure \ref{f_edge}
are the weights of the torus action.
\psset{unit=0.3 cm}
\begin{figure}[!htbp]
  \centering
   \begin{pspicture}(0,0)(14,10)
\psline{->}(3,5)(0,0)
\psline{->}(3,5)(0,10)
\psline(3,5)(11,5)
\psline{->}(11,5)(14,0)
\psline{->}(11,5)(14,10)
\rput[tr](-0.2,9.8){$t_1$}
\rput[br](-0.2,0.2){$t_2$}
\rput[tl](14.3,9.8){$t'_1$}
\rput[bl](14.3,0.2){$t'_2$}
\rput[bl](3.8,5.3){$t_3$}
\rput[tr](10.6,4.7){$-t_3$}
\end{pspicture}
 \caption{An edge in the toric polytope of $X$}
  \label{f_edge}
\end{figure}
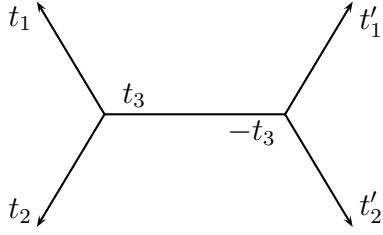

Let $e$ be marked by a partition $\lambda$. In Gromov-Witten
theory, such an edge corresponds to a $T$-fixed stable map
$$
f: C \to C_e \subset X\,,
$$
where the source curve
$C$ has a rational component $C_i$ for each part of $\lambda$. Restricted
to $C_i$, the map $f$ has the form
$$
\Pp\owns z \mapsto z^{\lambda_i} \in \Pp\,.
$$
The corresponding edge weight
$$
E_{GW}( \lambda,  t_1,t_2,t_3,t'_1,t'_2,u)
$$
equals $u^{-\chi(C)}$ times
the stack factor
$$
\zz(\lambda)^{-1} =
|\Aut f|^{-1}
$$
times the product of nonzero $T$-weights on
\begin{equation}\label{fred}
H^1(f^*(TX))\ominus H^0(f^*(TX)) 
\end{equation}
corrected by the $T$-weights of the tangent fields of $C$
vanishing at the end points.
The character of the virtual $T$-module \eqref{fred} may be easily
calculated for each component $C_i$ to be
\begin{equation}
  \label{edge_char}
\frac{e^{t_1} + e^{t_2}+e^{t_3}}{e^{-t_3/\lambda_i}-1}
+
\frac{e^{t'_1} + e^{t'_2}+e^{-t_3}}{e^{t_3/\lambda_i}-1} + 1 \,.
\end{equation}
Here we denote by $e^t\in \C[T]$ the irreducible
character corresponding to a weight $t\in\ft^*$.
The trivial character $1$, which is taken out in \eqref{edge_char},
is due to the tangent fields of $C$.

\subsubsection{Donaldson-Thomas edges}

In Donaldson-Thomas theory, a marked edge corresponds
to a $T$-fixed ideal sheaf $\cI_\lambda$ which restricts to a
monomial ideal of the form
\begin{equation}
I_\lambda = (x^i y^j)_{i\ge \lambda_{j+1}} \subset \C[x,y,z]
\label{mon_ideal}
\end{equation}
in the neighborhood of
each of the two $T$-fixed point. Here $z$ is the coordinate
along the edge.

Because of the asymmetry between
$x$ and $y$ in \eqref{mon_ideal}, the partition $\lambda$ should
properly be assigned to an oriented edge, with the change in
orientation resulting in the transposition of the partition.
One of the many advantages of \emph{capped localization}, which
will be introduced below, is absence of orientation issues.

By definition, the edge weight
$$
E_{DT}( \lambda, t_1,t_2,t_3,t'_1,t'_2,q)
$$
in Donaldson-Thomas theory
is $q^{\chi(\cO/\cI_\lambda)}$ times the product of $T$-weights on
\begin{equation}\label{bobb}
\Ext^2(\cI,\cI)\ominus \Ext^1(\cI,\cI)\,.
\end{equation}
Again, the character of the virtual $T$-module \eqref{bobb}
is easily
determined, see \cite{mnop1,mnop2}.

\subsection{Capped edges and vertices}\label{sCapV}

\subsubsection{Edges}
Let $e$ be a compact edge of the toric polytope corresponding to $X$.
Let $$X_e\subset X$$
 denote the non-compact toric variety 
 associated to $e$ determined by Figure 1.
Let 
$$F_0,F_\infty \subset X_e$$
be the $T$-invariant divisors lying over the
fixed points of the base $\Pp\subset X_e$.

By definition, the Gromov-Witten capped edge $e$,
\begin{equation}\label{cftt}
\bE_{GW}(\lambda,\mu,t_1,t_2,t_3,t_1',t_2',u)= \mathsf{Z}'_{GW}(X_e/
F_0 \cup F_\infty,u\ | \ \lambda,\mu),
\end{equation}
is the reduced $T$-equivariant partition function of
$X_e$ with free{\footnote{By free a relative condition,
we mean a partition with all cohomology weights 1.
Hence, just a partition.}} relative conditions $\lambda$
and $\mu$ imposed along $F_0$ and $F_\infty$ of degree
$$d=|\lambda|= | \mu |.$$
The $T$-fixed loci of the moduli space $\overline{M}'_{g}(X_e/F_0\cup F_\infty,d)$ are compact. Hence,
the partition function \eqref{cftt}
is well-defined by $T$-equivariant
residues, see \cite{BP}.

The definition of a capped edge in Donaldson-Thomas
 theory is identical,
\begin{equation}
\bE_{DT}(\lambda,\mu,t_1,t_2,t_3,t_1',t_2',u)= \mathsf{Z}'_{DT}(X_e/F_0
\cup F_\infty,q\ | \ \lambda,\mu),
\label{vvtt}
\end{equation}
again defined by $T$-equivariant residues.

The capped edge 
has normal bundle of type $(a,b)$ where
$$
(a,b) =  \left(\tfrac{t_1-t'_1}{t_3},\tfrac{t_2-t'_2}{t_3}\right) \in
\Z^2 \,.
$$
We will call the geometry here an  $(a,b)$-edge.
The toric variety $X_e$ is isomorphic to the
total space of the bundle
\begin{equation}\label{OaOb}
\mathcal{O}_{\Pp}(a)
\oplus \mathcal{O}_{\Pp}(b) \rightarrow \Pp.
\end{equation}

\subsubsection{Vertices}\label{s_def_CV}

Let $U$ be the $T$-invariant 3-fold obtained
by removing the three $T$-invariant  lines 
$$
L_1,L_2,L_3 \subset \Pp \times \Pp \times \Pp
$$
passing through the point $(\infty,\infty,\infty)$,
$$U = \Pp \times \Pp \times \Pp \ \setminus \ \cup_{i=1}^3 L_i.$$ 
Let $D_i \subset U$
be the divisor with $i^{th}$ coordinate $\infty$.
For $i\neq j$, the divisors ${D}_i$ and $D_j$
are disjoint.

In both Gromov-Witten and Donaldson-Thomas theories,
the capped vertex is the \emph{reduced}
partition function of $U$ with free relative conditions
imposed at the divisors ${D}_i$. 
While the relative geometry $U/\cup_i D_i$
is noncompact, the moduli spaces of maps 
$\overline{M}'_g(U/\cup_i D_i,\beta)$
and ideal sheaves $I_n(U/\cup_i D_i,\beta)$
have compact $T$-fixed loci.
The invariants of $U/\cup_i D_i$ in both theories are
well-defined by $T$-equivariant residues.
In the localization
formula for the reduced theories of 
 $U/\cup_i D_i$,
 nonzero degrees can occur {\em only} on the edges meeting
the origin $(0,0,0)\in Y$.

We
denote the capped GW vertex by
\begin{equation}\label{njp}
\bC_{GW}(\lambda,\mu,\nu,t_1,t_2,t_3,u) = \mathsf{Z}_{GW}'
(U/\cup_i D_i,u\ | \ \lambda,\mu,\nu)
\end{equation}
where $\lambda,\mu,\nu$ denote relative conditions imposed at $D_1,D_2,D_3$
and $t_1,t_2,t_3$ are the weights of the $T$-action on the
coordinate axes.

The definition of the capped DT vertex,
$$
\bC_{DT}(\lambda,\mu,\nu,t_1,t_2,t_3,q)
=
\mathsf{Z}_{DT}'(U/\cup_i D_i,q\ | \ \lambda,\mu,\nu),
$$
is parallel. The partitions $\lambda,\mu,\nu$ here represent the
Nakajima basis elements of
$$
H^*(\Hilb(D_i,[\beta]\cdot D_i),\mathbb{Q})\,, \quad i=1,\dots,3\,.
$$
Unlike the uncapped DT vertex, $\bC_{DT}$
enjoys a full $S(3)$ symmetry extending the obvious $S(3)$ action on
$X$.

\subsection{Capped localization}

\subsubsection{Overview}
\label{bgty}
Capped localization expresses the primary Gromov-Witten
 and Donaldson-Thomas
invariants of $X$ as a sum capped vertex and capped edge
data.

A half-edge $h=(e,v)$
 is a compact edge $e$ together with the
choice of an incident vertex $v$.
A partition assignment
$$h \mapsto \lambda(h)$$
to half-edges is {\em balanced} if the equality
$$|\lambda(e,v)| = |\lambda(e,v')|$$
always holds 
for the two halfs of  $e$.
For a balanced assignment, let $$|e|=|\lambda(e,v)|=|\lambda(e,v')|$$
denote the {\em edge degree}.

The outermost sum in the capped localization formula
runs over all balanced 
assignments of partitions $\lambda(h)$ to
the half-edges $h$ of $\Delta$ satisfying
\begin{equation}\label{ddfff}
\beta = \sum_e |e|\cdot \left[C_e\right] \in H_2(X,\mathbb{Z})\,.
\end{equation}
Such a partition assignment will be called a \emph{capped
marking} of
$\Delta$.
The weight of each capped marking in the localization sum for the
invariants \eqref{prgw} and \eqref{prdt} equals the product
of four factors:
\begin{enumerate}
\item[(i)] localization of the integrand,
\item[(ii)] capped vertex contributions,
\item[(iii)] capped edge contributions,
\item[(iv)] gluing terms.
\end{enumerate}

The integrand terms (i)  will be discussed in Section \ref{s_integr}
below. The integrand contributions in Gromov-Witten
and Donaldson-Thomas theories exactly match.
Each vertex determines up to three half-edges specifying the
partitions for the capped vertex.
Each compact edge determines two half-edges specifying
the partitions of the capped edge.
The gluing terms (iv) appear in exactly the same form
as in the degeneration formula.
Precise formulas are written in Section \ref{exxx}.

The capped localization formula
is easily derived from the standard localization formula.
Indeed, the capped objects are obtained from the
uncapped objects by  rubber integral{\footnote{
Rubber integrals $\langle \lambda \ |\ \frac{1}{1-\psi_\infty} \ |  
\ \mu \rangle ^\sim$ 
arise in the localization formulas for relative
geometries.
See
\cite{BP,GWDT} for detailed discussion of the 1-leg
rubber calculus. In particular, differential
equations governing the rubber integrals are discussed in
Section 11.2 of \cite{GWDT}.
}} factors.
The rubber integrals cancel in pairs in capped localization
to yield standard localization.

The GW/DT correspondence for
$X$ is a direct consequence of the capped localization
formula and the following results matching
the capped 
edges and vertices.

\begin{Proposition}\label{Ted} Capped edges satisfy the 
relative GW/DT correspondence. After
the substitution $q=-e^{iu}$, we have
\begin{multline*}
(-iu)^{d(a+b)+\sum_i \ell(\lambda^{(i)})}
\bE_{GW}(\lambda^{(1)},\lambda^{(2)},t_1,t_2,t_3) = \\
q^{-\frac{d(a+b+2)}{2}}
\bE_{DT}(\lambda^{(1)},\lambda^{(2)},t_1,t_2,t_3)\,
\end{multline*}
for an $(a,b)$-edge of degree $d$.
Capped edges are rational functions of $q$.
\end{Proposition}

\begin{Proposition}\label{Tccor} Capped vertices satisfy the 
relative GW/DT correspondence. After
the substitution $q=-e^{iu}$, we have
\begin{multline*}
(-iu)^{\sum_i |\lambda^{(i)}|+\ell(\lambda^{(i)})}
\bC_{GW}(\lambda^{(1)},\lambda^{(2)},\lambda^{(3)},t_1,t_2,t_3) = \\
q^{-\sum |\lambda^{(i)}|}
\bC_{DT}(\lambda^{(1)},\lambda^{(2)},\lambda^{(3)},t_1,t_2,t_3)\,.
\end{multline*}
Capped vertices are rational functions of $q$.
\end{Proposition}

Proposition \ref{Ted} has been proven in \cite{BP,GWDT}
for the vertical subtorus $$(\com^*)^2 \subset T.$$
The proof for the full 3-dimensional torus $T$ is given
in Section \ref{capcor}.
Proposition \ref{Tccor}, proven in Sections \ref{pr1}-\ref{pr2},
is really the
main result of the paper. 
Both Propositions \ref{Ted} and \ref{Tccor} are
special cases of the GW/DT correspondence for
$T$-equivariant residues conjectured in \cite{BP}.

Our proof of Proposition \ref{Tccor}
may be converted into an algorithm
for the computation of the capped vertices.  
Capped vertex evaluations in the first few
cases may be found in Section 4.3. 

From the perspective of the GW/DT correspondence,
the starred{\footnote{Here
we follow the notation of Section 9.2 of \cite{GWDT}.}}  
normalizations
\begin{eqnarray*}
\bE_{GW}^*(a,b)_{\lambda^{(1)},\lambda^{(2)}} &= &
(-iu)^{d(a+b)+\sum_i \ell(\lambda^{(i)})} 
\bE_{GW}(\lambda^{(1)},\lambda^{(2)},t_1,t_2,t_3)\ , \\
\bE_{DT}^*(a,b)_{\lambda^{(1)},\lambda^{(2)}} & = &
q^{-\frac{d(a+b+2)}{2}}
\bE_{DT}(\lambda^{(1)},\lambda^{(2)},t_1,t_2,t_3)
\end{eqnarray*}
for capped edges and 
\begin{eqnarray*}
\bC_{GW}^*(\lambda^{(1)},\lambda^{(2)},\lambda^{(3)}) & = & 
(-iu)^{\sum_i |\lambda^{(i)}|+\ell(\lambda^{(i)})}
\bC_{GW}(\lambda^{(1)},\lambda^{(2)},\lambda^{(3)},t_1,t_2,t_3)\ , \\
\bC_{DT}^*(\lambda^{(1)},\lambda^{(2)},\lambda^{(3)})
& = &
q^{-\sum |\lambda^{(i)}|}
\bC_{DT}(\lambda^{(1)},\lambda^{(2)},\lambda^{(3)},t_1,t_2,t_3)
\end{eqnarray*}
for capped vertices
will be more convenient.

\subsubsection{Localization of the integrand}\label{s_integr}

We  identify here the localization contributions of the
primary insertions. Let $v$ be a vertex of $\Delta$.
We will denote by the corresponding $T$-fixed point
by the same letter. Let $t_1,t_2,t_3$ be the
tangent weights at $v$ and let $d_1,d_2,d_3$ be
edge degrees determined by the capped marking.

Consider a primary insertion $\gamma\in H^*_T(X,\mathbb{Z})$.
Since the fixed points form a basis of $H^*_T(X,\mathbb{Z})$
after localization,
we may restrict our attention to
$$
\gamma = [v]\,.
$$
The pull-back  of $\gamma$ to $v$ is given by
$$
[v]\Big|_{v} = t_1 t_2 t_3 \,.
$$

\begin{Proposition}\label{llccv}
The localization of the incidence to $[v]$ in both
Gromov-Witten  and Donaldson-Thomas theories is
$$
t_1 t_2 d_3 + t_1 d_2 t_3 + d_1 t_2 t_3 \,.
$$
\end{Proposition}

  Since the Donaldson-Thomas techniques are not as well-known, we
derive the insertion term there. The corresponding
computation of the Gromov-Witten theory is a standard application of the
string equation.

We may assume that $X=\C^3$.
Consider the universal ideal sheaf \eqref{uIs}
and let
$$
\cI \subset \cO_X = \C[x_1,x_2,x_3]
$$
be a $T$-fixed ideal with degrees $d_i$ along
the coordinate axes. 
Proposition \ref{llccv} is an immediate consequence of 
the following result.

\begin{Lemma} \label{tyt}
We have
\begin{equation}
c_2(\fI)\Big|_{[\cI]\times v} = t_1 t_2 d_3 + t_1 d_2 t_3 + d_1 t_2 t_3
\,.\label{ttd}
\end{equation}
\end{Lemma}

\begin{proof}
Consider a graded free resolution of
length $3$,
\begin{equation}
0 \to R_3 \to R_2 \to R_1 \to \cI \to 0
\label{ResI}
\end{equation}
where
$$
R_i = \bigoplus_{k} x^{a_{ik}} \cO_X \,, \quad a_{ik} \in \Z^3 \,.
$$
Since $c_1(\fI)=0$, we have from \eqref{ResI}
\begin{equation}
c_2(\fI)|_{[\mathcal{I}]\times 0} = - \ch_2(\fI)|_{[\mathcal{I}]\times 0}
= \frac{1}{2}
\sum_{i,k} (-1)^i (t,a_{ik})^2 \,,
\label{c2fI}
\end{equation}
where $t=(t_1,t_2,t_3)$, and $(t, a)$ denotes the
standard inner product.

Let $\pi$ be a $3$-dimensional partition associated to $\cI$.
We may compute
the character of the $T$-action on
$\cO_X/\cI$ either directly in terms of $\pi$ or via the
resolution \eqref{ResI}. We obtain the identity
\begin{equation}
  \label{Trpi}
\prod_{i=1}^3 (1-e^{-t_i})\sum_{p\in \pi} e^{-(t,p)} =
1+ \sum_{i,k} (-1)^i e^{-(t, a_{ik})} \,.
\end{equation}
The right side of \eqref{c2fI} equals
the quadratic term in \eqref{Trpi}.
The sum over $p\in\pi$ of the left side of
\eqref{Trpi} may be split in 4 parts ---
the sums over $p$ in each of the three infinite legs of $\pi$
and the finite remainder. 
The infinite legs yield
geometric series which can be summed exactly.
Since
$$
\prod_{i=1}^3 (1-e^{-t_i}) = t_1 t_2 t_3 + O(t^4),
$$
the remainder does not contribute to the quadratic
term, while each infinite leg of $\pi$ contributes
the corresponding term to \eqref{ttd}.
\end{proof}

\subsubsection{Formulas}

\label{exxx}
Consider the 
capped localization formula 
for the partition function 
\begin{equation}\label{fq2}
\bZ'_{GW}\big(X,u\ \big|\ \gamma_1, \ldots, \gamma_r \big)_\beta
\end{equation}
with equivariant primary field insertions
$\gamma_{i} \in H_T^*(X,\mathbb{Z})$.

Let $\mathcal{V}$ be the set of vertices of $\Delta$. To each $v\in 
\mathcal{V}$,
let $h^{v}_1, h^v_2, h^v_2$ be the associated half-edges{\footnote{
For simplicity, we assume $X$ is projective so each vertex is
incident to 3 compact edges.}}
with tangent weights $t^v_1, t^v_2,t^v_3$ respectively.
Let $\Gamma_\beta$ be the set of capped markings
 satisfying the degree condition \eqref{ddfff}.
Each $\Gamma \in \Gamma_\beta$  associates  
a partition $\lambda(h)$ to every half-edge $h$. Let
$$|h| = |\lambda(h)|$$
 denote the half-edge degree.

The localization 
contributions of the insertions is easily identified. 
Let
$\gamma_i^v \in H^*_T(\bullet,\mathbb{Z})$
be the $T$-equivariant restriction of $\gamma_{i}$ to the
$T$-fixed point associated to $v$. 
The insertion term $\mathsf{I}_\Gamma$
is determined by
\begin{equation}\label{vrv}
\mathsf{I}_\Gamma =
\prod_{i=1}^r\ \sum_{v\in \mathcal{V}}
\gamma^v_i
\left(\frac{|h^v_1|}{t^v_1} +
\frac{|{h^v_2}|}{t^v_2}+
\frac{|{h^v_3}|}{t^v_3}\right)\ 
\in H^*_T(\bullet,\mathbb{Q})
\end{equation}
as a consquence of Proposition \ref{llccv}.

For each $v\in \mathcal{V}$, the assignment
$\Gamma$ determines an evaluation of the capped vertex,
$$\mathsf{C}_{GW}(v,\Gamma) = \mathsf{C}_{GW}(\lambda(h^v_1),
\lambda(h^v_2), \lambda(h^v_3), t^v_1, t^v_2, t^v_3, u).$$
Let
$h^e_1$ and $h^e_2$ be the half-edges
associated to the edge $e$.
The assignment 
$\Gamma$ also determines an evaluation of the capped edge,
$$\mathsf{E}_{GW}(e,\Gamma) = \mathsf{E}_{GW}(\lambda(h^e_1),
 \lambda(h^e_2), t_1,t_2,t_3,t'_1, t'_2   , u)$$
where the weights are determined by Figure 1.
A Gromov-Witten gluing factor is specified by $\Gamma$ at each
half-edge $h^v_i\in \mathcal{H}$ by 
$$\mathsf{G}_{GW}(h^v_i,\Gamma) = \mathfrak{z}(\lambda(h^v_i)) 
\left(\frac{\prod_{j=1}^3 t^v_j}{t^v_i}\right)^{\ell(\lambda(h^v_i))} 
u^{2\ell(\lambda(h^v_i))}.$$

The Gromov-Witten capped localization formula can be written
exactly in the form presented in Section \ref{bgty},
\begin{multline*}
\bZ'_{GW}\Big(X,u\ |\ \gamma_1, \ldots, \gamma_r
\Big)_\beta =\\
\sum_{\Gamma\in \Gamma_\beta}\
\prod_{v\in \mathcal{V}}\ \prod_{e\in \mathcal{E}}\ \prod_{h\in \mathcal{H}}
\mathsf{I}_\Gamma\
\mathsf{C}_{GW}(v,\Gamma) \
\mathsf{E}_{{GW}}(e,\Gamma)\
\mathsf{G}_{GW}(h, \Gamma)
\end{multline*}
where the product is over the sets of vertices $\mathcal{V}$, edges
 $\mathcal{E}$, and half-edges $\mathcal{H}$ 
of the polytope $\Delta$. 

The Donaldson-Thomas capped localization formula has an
identical structure,
\begin{multline*}
\bZ'_{DT}\Big(X,q\ |\ \gamma_1, \ldots, \gamma_r
\Big)_\beta =\\
\sum_{\Gamma\in \Gamma_\beta}\
\prod_{v\in \mathcal{V}}\ \prod_{e\in \mathcal{E}}\ \prod_{h\in \mathcal{H}}
\mathsf{I}_\Gamma\
\mathsf{C}_{DT}(v,\Gamma) \
\mathsf{E}_{{DT}}(e,\Gamma)\
\mathsf{G}_{DT}(h, \Gamma)
\end{multline*}
where the evaluations $\mathsf{C}_{DT}(v,\Gamma)$  and 
$\mathsf{E}_{{DT}}(e,\Gamma)$ are defined as before. 
 The Donaldson-Thomas gluing factors
$$\mathsf{G}_{DT}(h^v_i,\Gamma) = 
(-1)^{|h^v_i|-\ell(\lambda(h^v_i))}
\mathfrak{z}(\lambda(h^v_i)) 
\left(\frac{\prod_{j=1}^3 t^v_j}{t^v_i}\right)^{\ell(\lambda(h^v_i))} 
q^{-|h^v_i|}$$
are slightly different.

The most basic example of capped localization occurs
for the 3-fold total space of
\begin{equation}\label{bb23}
\cO(a) \oplus \cO(b) \rightarrow \Pp.
\end{equation}
The standard localization formula has vertices over
$0,\infty\in \Pp$ and a single 
edge.
To write the answer in terms of capped localization,
we consider a $T$-equivariant degeneration of  \eqref{bb23} to a chain
$$ (0,0) \cup (a,b) \cup (0,0)$$
of total spaces of bundles over $\Pp$
denoted here by splitting degrees.
The first $(0,0)$-geometry is relative over $\infty\in \Pp$,
the central $(a,b)$-geometry is relative on both sides, and
the last $(0,0)$-geometry is relative over $0\in \Pp$.
The degeneration formula exactly expresses the Gromov-Witten and 
Donaldson-Thomas theories of \eqref{bb23} as capped
localization with 2 capped vertices and a single capped
edge in the middle. 

In fact, the capped localization formula for arbitrary
toric $X$ in both theories
can be proven by studying the example \eqref{bb23} ---
the cancelling of the rubber caps already occurs there.
We leave the details
to the reader.

\section{Proofs of Propositions 1 and 2}\label{s_proof}

\subsection{Tube and cap invariants}

\subsubsection{Partition functions} \label{parfun}
Consider the 3-fold 
$
\bA_n \times \Pp$ with the full $T$-action.
The torus factors
$$T= (\com^*)^2 \times \com^*$$
into a vertical subtorus acting on $\bA_n$ and a
horizontal factor acting on $\Pp$.
There is a $T$-equivariant projection
$$
\pi:\bA_n\times \Pp \to \Pp \,.
$$
Let  $0,\infty\in\Pp$ be the $T$-fixed points.

We define 
partition functions in both Gromov-Witten
and Donaldson-Thomas theories with relative conditions
over the divisors
$$F_0, F_\infty \subset \bA_n\times\Pp$$
lying over $0,\infty \in \Pp$.
The geometry here is called the $\bA_n$-{\em tube}.
We set
\begin{multline}
\mathsf{GW}^*\left( \gamma_1,\dots,\gamma_r \right)_{\lambda,\mu} = \\
(-iu)^{\ell(\lambda)+\ell(\mu)}
\sum_{\sigma\in H_2(\bA_n,\mathbb{Z})} s^{\sigma} \,
\bZ'_{GW}\left( \bA_n\times \Pp
,u\, \big|\gamma_1,\dots,\gamma_r \big|
\lambda,\mu \right)_{d[\Pp]+\sigma} \,,
\label{bZhat}
\end{multline}
where $\lambda$ and $\mu$ represent relative conditions ---
 partitions of the horizontal degree $d$ labeled by elements
of $H^*_T(\bA_n,\mathbb{Z})$. 
The symbols $s^{\sigma}$ span the group ring
of $H_2(\bA_n,\Z)$. For brevity, we have dropped the
relative divisor $F_0\cup F_\infty$ in the notation.
Similarly, let
\begin{multline*}
\mathsf{DT}^*\left(\gamma_1,\dots,\gamma_r  \right)_{\lambda,\mu} = \\
(-q)^{d}
\sum_{\sigma\in H_2(\bA_n,\mathbb{Z})} s^{\sigma} \,
\bZ'_{DT}\left( \bA_n\times \Pp,q\, \big|\gamma_1,\dots,\gamma_r \big|
\lambda,\mu \right)_{d[\Pp]+\sigma} \,,
\end{multline*}

A benefit of the above notation is a
simple form for the conjectural GW/DT correspondence:
$$
\mathsf{GW}^*\left( \gamma_1,\dots,\gamma_r 
 \right)_{\lambda,\mu} =
\mathsf{DT}^*\left(\gamma_1,\dots,\gamma_r
 \right)_{\lambda,\mu}$$
after the variable change $e^{-iu}= -q$.

\subsubsection{Tube calculation}\label{tubcal}
Another advantage of the partition functions defined in
Section \ref{parfun} is their uniform behavior under
degeneration.

Let $g_{\lambda,\mu}$ be the natural $T$-equivariant
residue pairing
\begin{equation}
\label{innerp}
g_{\lambda,\mu} = \int_{\text{Hilb}(\bA_n,d)} \lambda \cup \mu
\end{equation}
where $|\lambda|=|\mu|= d$.
For example, for $S=\C^2$, we have
$$
g_{\lambda,\mu} =
 (-1)^{|\lambda|-\ell(\lambda)} 
 \frac{\delta_{\lambda,\mu}}{\zz(\lambda)} (t_1
 t_2)^{-\ell(\lambda)}\,.
$$
Define raised partition functions by
$$
\mathsf{GW}^*\left( \gamma_1,\dots,\gamma_r 
 \right)^{\lambda'}_{\mu} = \sum_\lambda
\mathsf{GW}^*\left( \gamma_1,\dots,\gamma_r 
 \right)_{\lambda,\mu} (-1)^d g^{\lambda,\lambda'},$$
$$\mathsf{DT}^*\left( \gamma_1,\dots,\gamma_r 
 \right)^{\lambda'}_{\mu} = \sum_\lambda
\mathsf{DT}^*\left( \gamma_1,\dots,\gamma_r 
 \right)_{\lambda,\mu} (-1)^d g^{\lambda,\lambda'}, \ \ \
$$
where $g^{\lambda,\mu}$ is the inverse matrix of the
intersection pairing.

Consider the $T$-equivariant degeneration of $\Pp$ into a union
$$\Pp \rightarrow \Pp \cup \Pp.$$
The degeneration formulas  for the geometry with {\em no insertions}
{\footnote{By convention, ${\mathsf{GW}^*}^{\lambda}_{\mu}= 
{\mathsf{GW}^*}(\ )^{\lambda}_{\mu}$ and similarly for
Donaldson-Thomas theory.}}
are
$${\mathsf{GW}^*}^{\lambda}_{\mu} = 
\sum_{\mu'}{\mathsf{GW}^*}^{\lambda}_{\mu'}
{\mathsf{GW}^*}^{\mu'}_{\mu}.$$
$${\mathsf{DT}^*}^{\lambda}_{\mu} = 
\sum_{\mu'}{\mathsf{DT}^*}^{\lambda}_{\mu'}
{\mathsf{DT}^*}^{\mu'}_{\mu}.$$
The $u=0$ and $s=0$ specialization (well-defined
for the starred normalization) on the Gromov-Witten
side shows the matrix
${\mathsf{GW}^*}^{\lambda}_{\mu}$
is invertible and hence equal to the identity.
The same conclusion is obtained on the Donaldson-Thomas
side via the specialization $q=0$ and $s=0$.
We conclude
$${\mathsf{GW}^*}^{\lambda}_{\mu} = {\mathsf{DT}^*}^{\lambda}_{\mu} 
= \delta^\lambda_\mu.$$

\subsubsection{Cap calculation}\label{scap}

The $\bA_n$-{\em cap} geometry is obtained by imposing relative
conditions on $\bA_n \times \Pp$ only along $F_\infty$ {\em with
no insertions}.
Define 
\begin{eqnarray*}
\mathsf{GW}^*_{\lambda} & = &
(-iu)^{d+\ell(\lambda)}
\sum_{\sigma\in H_2(\bA_n,\mathbb{Z})} s^{\sigma} \,
\bZ'_{GW}\left( \bA_n\times \Pp,u\, \big|\  \big|
\lambda \right)_{d[\Pp]+\sigma} \ , \\
\mathsf{DT}^*_{\lambda} & = &
(-q)^{d}
\sum_{\sigma\in H_2(\bA_n,\mathbb{Z})} s^{\sigma} \,
\bZ'_{DT}\left( \bA_n\times \Pp,q\, \big|\ \big|
\lambda \right)_{d[\Pp]+\sigma} \,.
\end{eqnarray*}
The formulas for raising indices are the same in the
tube case.

Consider the cap geometry for $S \times \Pp$ where
$S$ is a compact surface and
$K_S\cdot \sigma \le 0$ for all curve classes $\sigma\in H_2(S,\mathbb{Z})$
of interest. 
Most cap invariants
vanish for dimension reasons. 
Indeed, if $\beta=d[\Pp]+\sigma$, the
virtual dimension
$$
\bd =2d - \sigma \cdot K_S
$$
is bounded below by
the dimension $2d$ of $\Hilb(S,d)$, which is also the
maximal codimension of relative conditions in Gromov-Witten theory.
We must have $\sigma \cdot K_S = 0$ 
and $\lambda \in  
H^{0}(\Hilb(S,d),\mathbb{Q})$ for
the cap invariants ${\mathsf{GW}^*_S}^\lambda$
and ${\mathsf{DT}^*_S}^\lambda$ with {\em no insertions}
 to have chance of
not
vanishing.

Let $\bA_n \subset S$ be a $T$-equivariant embedding in
a compact toric surface $S$.
On the Gromov-Witten side, the moduli of maps to 
$\bA_n \times \Pp$ occurs as an open and closed
subset of the moduli of maps to $S\times \Pp$.
On the Donaldson-Thomas side,
the dimension 0 components of the subschemes can wander off
of $\bA_n \times \Pp$, but a localization argument
shows the reduced theory{\footnote{The division
by the degree 0 Donaldson-Thomas theory
removes contributions of the complement of
$\bA_n$ in $S$.}} of $\bA_n\times \Pp$ is determined by an
open and closed  locus of the moduli of ideal sheaves
of $S\times \Pp$.

By considering the inclusion $\bA_n\subset S$, we draw
two conclusions. First, the cap invariants
$${\mathsf{GW}^*}^\lambda = {\mathsf{DT}^*}^\lambda = 0$$
vanish unless $\lambda$ is the unique partition $1^d$ with
each part weighted by the identity in $H^0(\bA_n,\mathbb{Z})$.
Let us denote the latter partition by $\mathbf{1}^d$.
Second, the cap invariants
\begin{equation}\label{nby}
{\mathsf{GW}^*}^{\mathbf{1}^d}, \ \ {\mathsf{DT}^*}^{\mathbf{1}^d}
\end{equation}
are {\em non-equivariant scalars}.
Hence, the $T$-equivariant integrals \eqref{nby} can be calculated via 
a restricted torus. For the vertical subtorus $(\com^*)^2$,
the calculations of \cite{dmt,mo2} yield
\begin{equation*}
{\mathsf{GW}^*_S}^{\mathbf{1}^d} = {\mathsf{DT}^*_S}^{\mathbf{1}^d}=1.
\end{equation*}

\subsection{Differential equations}\label{sDE}

\subsubsection{Equation I}


Let $t_3$ be the weight of the $T$-action on
the trivial normal bundle of
$$F_0 \subset \bA_n \times
\Pp.$$
By localization, every class $\gamma\in H^*_T(\bA_n\times \Pp,\mathbb{Z})$ 
satisfies
\begin{equation}\label{bhhr}
t_3 \, \gamma  = \gamma_0 - \gamma_\infty \ \in H^*_T(\bA_n\times \Pp,
\mathbb{Z})
\end{equation}
where 
$\gamma_0,\gamma_\infty$ are the
restrictions of $\gamma$ to $F_0$ and $F_\infty$
respectively.

We may express the insertion of the class $\gamma_0$
in terms of rubber integrals via the following gluing formula.
For the Gromov-Witten theory of $\bA_n \times \Pp$,
\begin{equation}
\lang \lambda \big| \gamma_0 \big| \mu \rang'_\beta
= \sum_{\nu,\,\sigma_1+\sigma_2 = \sigma}
\lang \lambda \big| \gamma_0 \big| \nu \rang^\sim_{\beta_1}{}' \,
\zz(\nu) \, u^{2\ell(\nu)} \,
\lang \nu^\vee \big|  \ \big| \mu \rang'_{\beta_2}
\label{gamma0}
\end{equation}
where
$$
\beta_i =  d [\Pp]+\sigma_i \,, \quad \sigma_i \in H_2(\bA_n,\mathbb{Z})\,,
$$
is the decomposition of the curve class into an 
$\bA_n$-component and
a $\Pp$-component. The sum in \eqref{gamma0}
is over all intermediate relative conditions
$\nu$ and all splittings of the $\bA_n$-component 
$\sigma \in H_2(\bA_n,\mathbb{Z})$ of the degree $\beta$.
The bracket on the left denotes the
theory of $\bA_n \times \Pp$ relative to $F_0 \cup F_\infty$ with
relative conditions  $\lambda, \mu$ and insertion $\gamma_0$.
The first set of brackets on the right with a tilde superscript
denotes  the {rubber} theory{\footnote{See \cite{mo2,GWDT}
for a foundational discussion of rubber theory in the
$\bA_n$ context.}}
 of $\bA_n \times \Pp$ with the insertion
$\gamma_0$ pulled back from $\bA_n$. The second bracket on the
right is the $\bA_n$-tube.

The proof of \eqref{gamma0} 
is obtained from basic geometry
of the relative theory. 
Let $\mathcal{D}$ be the $T$-equivariant Artin stack of degenerations
of the 1-pointed relative geometry  $\bA_n\times \Pp/
F_0 \cup F_\infty$.
The moduli space $\mathcal{D}$ parameterizes 
accordian destabilizations $Y$
of $\bA_n \times \Pp$ together with a point $p\in Y$
not on the boundary or the singular locus.
There is a $T$-equivariant  evaluation map
$$\text{ev}_p: \mathcal{D} \rightarrow \bA_n \times \Pp.$$
There are two divisors on $\mathcal{D}$ related to $0\in \Pp$.
The first is $\text{ev}_p^{*}([F_0])$. The second is the
boundary divisor
$\mathcal{D}_0\subset \mathcal{D}$
where $p$ is on a destabilization over $0\in \Pp$.
The following result can be easily seen by comparing the divisors
on smooth charts for $\mathcal{D}$.

\begin{Lemma}\label{nfq}
 $ \text{\em ev}_p^*([F_0]) = \mathcal{D}_0$ in
$\text{\em Pic}(\mathcal{D}).$
\end{Lemma}

Equation \eqref{gamma0} is proven by
factoring 
$$\gamma_0= \gamma \cdot [F_0],$$ pulling-back 
Lemma \ref{nfq} to the moduli space of stable maps
 to
 $\bA_n\times \Pp/
F_0 \cup F_\infty$, and using the splitting formulas
\cite{L}.

The parallel result in Donaldson-Thomas theory holds with
DT gluing factors,
\begin{equation}
\lang \lambda \big| \gamma_0 \big| \mu \rang'_\beta
= \sum_{\nu,\,\sigma_1+\sigma_2 = \sigma}
\lang \lambda \big| \gamma_0 \big| \nu \rang^\sim_{\beta_1}{}' \,
(-1)^{|\nu|-\ell(\nu)} \,
\zz(\nu) \, q^{-|\nu|} \,
\lang \nu^\vee \big|  \ \big| \mu \rang'_{\beta_2}
\label{gamma00}
\end{equation}
 The proof is identical. The
universal relative space over the moduli space of
relative ideal sheaves discussed in \cite{GWDT} is used.

\subsubsection{Equation II}
\label{twtw}
Let $\gamma\in H^*_T(\bA_n\times\Pp,\mathbb{Z})$ 
be the pull-back of
a divisor $\Gamma\subset \bA_n$. By the divisor equation, the insertion of
$\gamma$ simply multiplies an invariant in class $d[\Pp]+\sigma$
by
$$
\gamma\cdot \beta = \Gamma \cdot \sigma \,.
$$
The insertion of $\gamma$ may thus be interpreted as the action of
a linear differential operator $\partial_\Gamma$ on the
generating function \eqref{bZhat} over all possible
vertical curve classes $\sigma$.

The formulas \eqref{gamma0}-\eqref{gamma00}
may be interpreted as left multiplication by the matrix
corresponding to
\begin{equation}
  \label{gamma_pt}
  \lang \lambda \big| \gamma_0 \big| \nu \rang^\sim_{\beta}{}' =
\lang \lambda \big|   \Gamma_F \big| \nu \rang'_{\beta}{} \,.
\end{equation}
Here $\Gamma_F = \gamma \cdot F$, where
$$
F \in H^2_{(\com^*)^2}(\bA_n \times \Pp)\,,
$$
is the fiber class --- the pull-back
of the generator of the nonequivariant
group $H^2(\Pp,\Z)$. 
The nonequivariance
here is because equation \eqref{gamma_pt} arises
from rigidification of the rubber.
The brackets on the
right side of \eqref{gamma_pt} take values in
the fiberwise $(\com^*)^2$-equivariant
cohomology.{\footnote{There is no possibility for confusion
in equation \eqref{gamma_pt} because even if
$\lang \lambda \big|   \Gamma_F \big| \nu \rang'_{\beta}{}$
is interpreted as a $T$-equivariant
bracket, the result is independent of
the $t_3$ along the $\Pp$ direction. The proof is
left to the reader.}}

We can write the formulas \eqref{gamma0}-\eqref{gamma00}
as
$${\mathsf{GW}^*}(\gamma_0)^\lambda_\mu
=\sum_\nu {\mathsf{GW}^*}(\Gamma_F)^\lambda_\nu
\ {\mathsf{GW}^*}^\nu_\mu\ ,$$
$${\mathsf{DT}^*}(\gamma_0)^\lambda_\mu
=\sum_\nu {\mathsf{DT}^*}(\Gamma_F)^\lambda_\nu
\ {\mathsf{DT}^*}^\nu_\mu \ .$$
By \eqref{bhhr}, we obtain the following linear
differential equations in the Gromov-Witten theory of $\bA_n \times \Pp$,
\begin{equation}
  \label{divisor}
  t_3 \, \partial_\Gamma  {\bO_{GW}}  = \big
[{\bO_{GW}}(\Gamma_F), {\bO_{GW}}],  
\end{equation}
Here, we consider the invariants ${\mathsf{GW}^*}_\bullet^\bullet$
as defining an operator $\bO_{GW}$ on the Fock space
associated to $\bA_n$,
\begin{equation} \label{xxy}
\bO_{GW} |\mu \rangle = \sum_\lambda {\mathsf{GW}^*}^\lambda_\mu |
\lambda\rangle,
\end{equation}
or equivalently
$$\langle\lambda| \bO_{GW} | \mu \rangle = (-1)^d
{\mathsf{GW}^*}_{\lambda,\mu}$$
where the inner product on Fock space
defined by \eqref{innerp} occurs on the left side.
The operator $\bO_{GW}(\Gamma_F)$ is defined similarly.
The identical equation 
\begin{equation}
  \label{divisordt}
  t_3 \, \partial_\Gamma  {\bO_{DT}}  = \big
[{\bO_{DT}}(\Gamma_F), {\bO_{DT}}],  
\end{equation}
holds
in Donaldson-Thomas theory.

Unfortunately, equations \eqref{divisor} and \eqref{divisordt}
 are trivial here.
We have already seen
$\bO_{GW}$ is the identity matrix.
In particular, only $\sigma=0$ terms occur. Hence, the
left sides of \eqref{divisor} and \eqref{divisordt} vanish. The right sides
of \eqref{divisor} and \eqref{divisordt}
 are also 0 since commutators with
the identity matrix always vanish.
However,
we will make nontrivial use of the differential 
equations 
in more complicated geometries in the remainder of the
paper.

The calculations of 
$\bO_{GW}(\Gamma_F)$ and 
$\bO_{DT}(\Gamma_F)$ are 
certainly not formal. 
The equality
we need is Corollary 8.5 of \cite{mo2}.

\begin{Theorem}\label{mmoo} 
For $\bA_n\times \Pp$,
$${\mathsf{GW}^*}(\Gamma_F)_\lambda^\mu = 
{\mathsf{DT}^*}(\Gamma_F)_\lambda^\mu$$
after the variable change $e^{iu}=-q$,
\end{Theorem}

Theorem \ref{mmoo} is a special case of
the equivariant relative GW/DT correspondence for 
toric varieties.{\footnote{More of the correspondence
is proven in \cite{dmt,mo1,mo2}, but Theorem \ref{mmoo} is
all we will require.}}
The result will play a crucial
role for us.

\subsection{Proof of Proposition 1}
\label{capcor}

\subsubsection{Fiberwise $(\com^*)^2$-action}

Since capped edges are reduced partition functions in the
Gromov-Witten and Donaldson-Thomas
 theories of local curves, the results of \cite{BP,OP6,GWDT}
establish the GW/DT correspondence for the fiberwise $(\C^*)^2$-equivariant
cohomology.
Our goal now is
to strengthen the correspondence to include the full
$3$-dimensional $T$-action.

\subsubsection{$(0,0)$ and $(0,-1)$-edges}

The $(0,0)$-edge is the theory of 
$\com^2 \times \Pp$ 
relative to fibers over $0,\infty \in \Pp$
with respect to the full 3-dimensional $T$-action.
The geometry here is just the $\bA_0$-tube, so we have already
proven the required GW/DT correspondence in Section \ref{tubcal}.

By standard $T$-equivariant degeneration arguments,
the correspondence for all $(a,b)$-edges follows from the
$(0,0)$,$(0,-1)$, and $(-1,0)$ cases.
By symmetry, we need only consider the $(0,-1)$-edge,
the total space of 
$$\cO \oplus \cO(-1) \rightarrow \Pp$$
with respect to the full $T$-action.

The starred
partition functions for the $(0,-1)$-edge are
\begin{eqnarray*}
\bE_{GW}^*\left(0,-1\right)_{\lambda,\mu} &= & 
(-iu)^{\ell(\lambda)+\ell(\mu)-d}\
\bZ'_{GW}\left((0,-1),u\ | \ \lambda,\mu \right), \\
\bE_{DT}^*\left(0,-1\right)_{\lambda,\mu} &= & 
(-q)^{-\frac{d}{2}}\ 
\bZ'_{DT}\left((0,-1),q\ | \ \lambda,\mu\right).
\end{eqnarray*}
The GW/DT correspondence we require then takes the form
$$\bE_{GW}^*\left(0,-1\right)_{\lambda,\mu}
=
\bE_{DT}^*\left(0,-1\right)_{\lambda,\mu}$$
after the variable change $e^{-iu}=-q$.

\subsubsection{DT  descendents}
Primary insertions in Donaldson-Thomas theory were
defined in \eqref{c2oper} via the
the K\"unneth components of $c_2(\fI)$. Of course,
we may also consider the other characteristic classes of the
universal ideal sheaf $\fI$.

Let $X$ be a nonsingular 3-fold.
Following the terminology of \eqref{c2oper},
we define the
{\em descendent} insertion $\sigma_k(\gamma)$ by
the operation
\begin{equation*}
\sigma_{k}(\gamma)\big(\xi\big)=
 \pi_{1*}\big(\ch_{k+2}(\fI)
\cdot \pi_2^*(\gamma) \cap      \pi_1^*(\xi)\big)\,.
\end{equation*}
The insertion $\sigma_k(\gamma)$
 lowers the (real) homological degree by $2k+\deg\gamma - 2$.
In particular, $\sigma_1(1)$ preserves the degree. By Riemann-Roch,
$$
\sigma_1(1) = - \chi+ \frac{1}{2} \int_\beta c_1(X)\,,
$$
where $\chi=\chi(\cO/\cI)$ and $\beta$ is the curve class.

Specializing now to the $(0,-1)$-edge geometry, we see
\begin{equation}\label{gthy}
\bZ_{DT}((0,-1),q \ | \ \sigma_1(1)\ |\ 
\lambda,\mu) = (-q \frac{d}{dq} + \frac{d}{2})\ 
\bZ_{DT}((0,-1),q\ |\ \lambda,\mu)
\end{equation}
where the partition functions are {\em unprimed}.
The degree 0 series \cite{GWDT} is
$$\bZ_{DT}((0,-1),q\ |\ \emptyset,\emptyset)=
M(-q)^{-\frac{t_1+t_2}{t_3}} \cdot M(-q)^{\frac{t_1'+t'_2}{t_3}}
$$
where the weights $t_i$ are specified by Figure 1 and
$$M(-q) = \prod_{n=1}^\infty \frac{1}{(1-(-q)^n)^n}$$
is the McMahon function \eqref{macf}

After transforming \eqref{gthy}, we obtain the following
equation
\begin{multline} \label{kqqq}
-q \frac{d}{dq} \ \bE_{DT}^*(0,-1)^{\lambda}_{\mu} =\\
\bE_{DT}^*(0,-1 \ |\ \sigma_1(1))^{\lambda}_\mu -
\left(\frac{t_1+t_2}{t_3}-\frac{t'_1+t'_2}{t_3}\right) \Phi(q)
\bE_{DT}^*(0,-1)^{\lambda}_\mu
\end{multline}
where
$$\bE_{DT}^*(0,-1 \ |\ \sigma_1(1))_{\lambda,\mu} =
(-q)^{-\frac{d}{2}} \frac
{\bZ_{DT}((0,-1),q \ | \ \sigma_1(1)\ |\ \lambda,\mu)}
{\bZ_{DT}((0,-1),q\ |\ \emptyset,\emptyset)}
$$
and
$$\Phi(q)= q \frac{d}{dq} M(-q).$$
Next, we use the relation
$$t_3 \sigma_1(1)= \sigma_1(F_0) - \sigma_1(F_\infty)$$
and the differential equation of Section \ref{sDE} to conclude
\begin{eqnarray*}
t_3\bE_{DT}^*(0,-1 \ |\ \sigma_1(1))^{\lambda}_{\mu}
&=& \ \ \ 
\sum_\nu \bE_{DT,t_1,t_2}^*(0,0 \ | \ \sigma_1(F))^{\lambda}_\nu
\ \bE_{DT}^*(0,-1)^{\nu}_{\mu}\\
& & -
\sum_\nu \bE_{DT}^*(0,-1)^{\lambda}_{\nu} \
\bE_{DT,t_1',t_2'}^*(0,0 \ | \ \sigma_1(F))^{\nu}_\mu\ .
\end{eqnarray*}
where
$$\bE_{DT,w_1,w_2}^*(0,0 \ | \ \sigma_1(F))_{\lambda,\mu} =
(-q)^{-d} \frac
{\bZ_{DT}((0,0),q \ | \ \sigma_1(F)\ |\ \lambda,\mu)}
{\bZ_{DT}((0,0),q\ |\ \emptyset,\emptyset)}$$
and the subscripted weights specify the fiberwise $T$-action.

Written as operators on the Fock space of $\bA_0$, we
obtain
\begin{eqnarray*}
-t_3 q \frac{d}{dq} \bO_{DT} (0,-1)& = & \ \  \  
\bO^{t_1,t_2}_{DT} ((0,0)\ | \ \sigma_1(F) )\ \bO_{DT} (0,-1) \\
& & -\bO_{DT} (0,-1)\ \bO_{DT}^{t_1',t_2'} ((0,0)\ | \ \sigma_1(F) ) \\
& & - \left({t_1+t_2}-{t'_1-t'_2}\right) \Phi(q) 
\bO_{DT} (0,-1).
\end{eqnarray*}
By Proposition 22 of \cite{GWDT},
$$\bO^{w_1,w_2}_{DT} ((0,0)\ | \ \sigma_1(F) )=
 - \mathsf{M}(w_1,w_2) + (w_1+w_2) \Phi(q) \cdot \text{Id},$$
where $\mathsf{M}$ is the fundamental operator{\footnote{We
follow here the terminology of \cite{GWDT}
for the raising and lowering operators $\alpha_{\pm r}$
on Fock space.}} on
Fock space defined by
\begin{multline*}
\mathsf{M}(w_1,w_2) = (w_1+w_2) \sum_{k>0} \frac{k}{2} \frac{(-q)^k+1}{(-q)^k-1} \,
 \alpha_{-k} \, \alpha_k  + \\
\frac12 \sum_{k,l>0} 
\Big[w_1 w_2 \, \alpha_{k+l} \, \alpha_{-k} \, \alpha_{-l} -
 \alpha_{-k-l}\,  \alpha_{k} \, \alpha_{l} \Big] \,.
\end{multline*}
We conclude
\begin{equation} \label{mmmq}
-t_3 q \frac{d}{dq} \bO_{DT} (0,-1) =
-\mathsf{M}(t_1,t_2) \ \bO_{DT} (0,-1)
+\bO_{DT} (0,-1) \mathsf{M}(t_1',t_2')
\end{equation}
holds.

\subsubsection{GW descendents}
The dilaton equation for the 
descendent insertion $\tau_1(1)$ 
yields
$$\left(u\frac{d}{du}+d\right) \ \bE_{GW}^*(0,-1)^\lambda_\mu =
\bE_{GW}^*((0,-1)\ | \tau_1(1))^\lambda_\mu$$
where
$$\bE_{GW}^*(0,-1)_{\lambda,\mu} = (-iu)^{\ell(\lambda)+\ell(\mu)-d}
\mathsf{Z}_{GW}'(0,-1)_{\lambda,\mu},$$
$$\bE_{GW}^*((0,-1)\ | \ \tau_1(1))_{\lambda,\mu} = 
(-iu)^{\ell(\lambda)+\ell(\mu)-d}
\mathsf{Z}_{GW}'((0,-1)\ | \ \tau_1(1))_{\lambda,\mu}.$$
Next, we use the relation
$$t_3 \tau_1(1)= \tau_1(F_0) - \tau_1(F_\infty)$$
and the differential equation of Section \ref{sDE} to conclude
\begin{eqnarray*}
t_3\bE_{GW}^*(0,-1 \ |\ \tau_1(1))^{\lambda}_{\mu}
&=& \ \ \ 
\sum_\nu \bE_{GW,t_1,t_2}^*(0,0 \ | \ \tau_1(F))^{\lambda}_\nu
\ \bE_{GW}^*(0,-1)^{\nu}_{\mu}\\
& & -
\sum_\nu \bE_{GW}^*(0,-1)^{\lambda}_{\nu} \
\bE_{GW,t_1',t_2'}^*(0,0 \ | \ \tau_1(F))^{\nu}_\mu\ ,
\end{eqnarray*}
where
\begin{equation}\label{v89}
\bE_{GW,w_1,w_2}^*(0,0 \ | \ \tau_1(F))_{\lambda,\mu} =
(-iu)^{\ell(\lambda)+\ell(\mu)} 
{\bZ_{GW}'((0,0),q \ | \ \tau_1(F)\ |\ \lambda,\mu)}
\end{equation}
and the subscripted weights specify the fiberwise $T$-action.

Putting the results together, we obtain the main equation,
\begin{eqnarray*} 
t_3
\left(u\frac{d}{du}+d\right) \ \bE_{GW}^*(0,-1)^\lambda_\mu
&=& \ \ \
\sum_\nu \bE_{GW,t_1,t_2}^*(0,0 \ | \ \tau_1(F))^{\lambda}_\nu
\ \bE_{GW}^*(0,-1)^{\nu}_{\mu}\\
& & -
\sum_\nu \bE_{GW}^*(0,-1)^{\lambda}_{\nu} \
\bE_{GW,t_1',t_2'}^*(0,0 \ | \ \tau_1(F))^{\nu}_\mu\ .
\end{eqnarray*}
The change of variables $e^{iu}=-q$ implies
$$-t_3 q \frac{d}{dq} = -\frac{1}{iu} \left(t_3 u\frac{d}{du}\right).$$
By a straightforward{\footnote{The evaluation
of the fiberwise $(\com^*)^2$-equivariant Gromov-Witten integral
\begin{equation}\label{tttyyq}
\langle \lambda \ | -\tau_1(F) \ | \ \mu \rangle^{(0,0)},
\end{equation}
determining \eqref{v89} proceeds in several well-known steps.
First, the series \eqref{tttyyq} is related to 
\begin{equation}\label{gglle}
\langle \lambda \ | -(2,1^{d-2}) \ | \ \mu \rangle^{(0,0)}
\end{equation}
by degeneration --- the parallel step in Donaldson-Thomas
theory is done in Section 9.1 of \cite{GWDT}.
The evaluation of \eqref{gglle} is a central result of \cite{BP}.
The difference between \eqref{tttyyq} and \eqref{gglle} is very easily
evaluated. The only new integral which must be computed
is $$\int_{[\overline{M}_{g,1}(\Pp,1)]^{vir}} \lambda_g\lambda_{g-1}
\tau_1(p),$$
where $p\in H^2 (\Pp,\mathbb{Z})$ is the point class.
By localization, the integral is immediately reduced
to the Hodge integral series
$$ \sum_{g\geq 1}2g \cdot  u^{2g}   \int_{\overline{M}_{g,1}}
\lambda_g \lambda_{g-1} \frac{ \sum_{i=0}^{g} (-1)^i \lambda_i}
{1-\psi_1}
 =u\frac{d}{du}\log \frac{u/2}{\sin(u/2)}$$ 
calculated in \cite{pan}. We leave the details to the
reader.}}
evaluation of \eqref{v89},
we find $-\frac{1}{iu}$ times  the main equation can be written
in $q$ as
\begin{equation} \label{mmmqq}
-t_3 q \frac{d}{dq} \bO_{GW} (0,-1) =
-\mathsf{M}(t_1,t_2) \ \bO_{GW} (0,-1)
+\bO_{GW} (0,-1) \mathsf{M}(t_1',t_2') \ .
\end{equation}
The term on the left 
$$-\frac{1}{iu} \cdot t_3 d \bO_{GW}(0,-1)$$
is exactly cancelled by the differences between \eqref{v89}
and the two instances of $\mathsf{M}$ on the right side.
Hence, we have an exact match with \eqref{mmmq}.

\subsubsection{Matching}
Let $\cV\subset \cF$ be the linear subspace of the Fock
space of $\bA_0$ consisting of vectors
$\left|v\rang$ for which the 
$(0,-1)$-edge matrix element
$$
\lang \lambda \big| \bO_{DT}(0,-1) \big|v\rang \in \Q(t_1,t_2,t_3)((q))
$$
is a rational function of $q$ satisfying the relative GW/DT
correspondence \eqref{relcor} for all $\lambda$.

We first prove $\cV$ is nonempty by showing 
$$
\left|1^d\rang \in \cV
$$
for all $d$. Indeed,
\begin{equation}\label{hyy7}
\lang \lambda \big| \bO_{DT}(0,-1) \big|1^d\rang =
t_1^{-\ell(\lambda)} \, \lang \lambda[L_{-1}] \big| \bO_{DT}(0,-1)
\big|1^d\rang'
\end{equation}
where the cohomology label $[L_{-1}]$ is the Poincar\'e dual of
the $\cO(-1)$-axis in the fiber over $0\in\Pp$.
The relative conditions imply the bracket on the right is an
integral of the correct dimension over a proper space
(modulo point contributions
 removed in the reduced invariant, see
Section 10.3 of \cite{GWDT}).
The bracket on the right is thus
{\em independent} of the equivariant parameters. 
The same conclusion holds in Gromov-Witten theory.

The GW/DT correspondence for \eqref{hyy7} is the same for
$T$-equivariant cohomology and  $(\com^*)^2$-equivariant
cohomology since the answer is weight independent.
Since the $(\com^*)^2$-equivariant statement has been proven,
$T$-equivariant statement also holds.

Using the identical equations \eqref{mmmq} and \eqref{mmmqq} for the 
two theories, we conclude 
$\cV\otimes \Q(q,t_1,t_2,t_3)$ is closed under the action of the operator
$$
\nabla = t_3 \,q \frac{d}{dq}
- \MM(t'_1,t'_2) \,.
$$
Let $p(d)$ denote the number of partitions of $d$. We claim
the vectors
$$
\nabla^k \, \left|1^d\rang\,, \quad k=0,\dots,p(d)-1
$$
are linearly independent over $\Q(q,t_1,t_2,t_3)$ and therefore span the
entire subspace $\cF_d \subset \cF$ of vectors of energy $d$.
In fact, the $t_3$-constant terms of these vectors
are already linearly independent, see \cite{BP,OP6}.
Hence,
$\cV$ is the entire Fock space $\cF$, and the proof
of Proposition 1 is complete.
\qed

\subsection{Capped rubber}\label{s_capr}

We start with
the tube,
$$
\pi: \bA_n \times \Pp \rightarrow \Pp
$$
relative to the fibers over $0,\infty\in \Pp$. 
In the moduli space
of stable maps, we define a
$T$-equivariant open subset
$$U^{CR}_{g,\beta} \subset \overline{M}'_{g}(\bA_n\times\Pp/F_0\cup F_\infty,
\beta)$$
consisting of map
with {\em no positive degree components}
in the destabilization of the fiber over $0\in\Pp$.
The open set
$$V^{CR}_{n,\beta} \subset {I}_{n}(\bA_n\times\Pp/F_0\cup F_\infty,
\beta)$$
is defined in exactly the same way.
The $T$-equivariant residue theories of $U^{CR}$ and $V^{CR}$
are well-defined since the $T$-fixed loci are compact.
The geometry here is called the {\em capped $\bA_n$-rubber}.{\footnote{
Another approach to capped rubber in Gromov-Witten and
Donaldson-Thomas theory is to start with $A_n$-rubber
\cite{dmt,mo2}
and then add $n$ degree 0 caps. The result is precisely
what is obtained by localization on $U^{CR}$ and $V^{CR}$.}}

We define 
partition functions with relative conditions
over the divisors
$F_0, F_\infty$
lying over $0,\infty \in \Pp$.
We set{\footnote{The symbols $U^{CR}$ and $V^{CR}$ here indicates not the targets but rather the
open sets of the moduli spaces associated to $\bA_n\times \Pp$.}}
\begin{eqnarray}
\mathsf{GW}^{CR*}_{\lambda,\mu} & = & 
(-iu)^{\ell(\lambda)+\ell(\mu)}
\sum_{\sigma\in H_2(\bA_n,\mathbb{Z})} s^{\sigma} \,
\bZ'_{GW}\left( U^{CR}
,u\, \big|\ \big|
\lambda,\mu \right)_{d[\Pp]+\sigma} \,, \nonumber\\\label{vee3}
\mathsf{DT}^{CR*}_{\lambda,\mu} & =& 
(-q)^{d}
\sum_{\sigma\in H_2(\bA_n,\mathbb{Z})} s^{\sigma} \,
\bZ'_{DT}\left( V^{CR},q\, \big|\ \big|
\lambda,\mu \right)_{d[\Pp]+\sigma} \,.
\end{eqnarray}
following the conventions of Section \ref{parfun}.
The GW/DT correspondence takes the form
$$
\mathsf{GW}^{CR*}_{\lambda,\mu} =
\mathsf{DT}^{CR*}_{\lambda,\mu}$$
after the variable change $e^{-iu}= -q$.

\begin{Lemma}\label{Pcapr}
The capped $\bA_n$-rubber invariants are
 rational functions of $q$ and
satisfy the  $GW/DT$ correspondence.
\end{Lemma}

\begin{proof}
The capped rubber invariants \eqref{vee3} determine operators
$$\bO_{GW}(\CR),\ \bO_{DT}(\CR): \mathcal{F} \rightarrow \mathcal{F}$$
acting on the 
the Fock space associated to $\bA_n$ following Section \ref{twtw}.

Equations \eqref{gamma0}-\eqref{gamma00}
 remain valid in the residue theory with the
restriction $\sigma_1=0$ by the definition of
$U_{CR}$ and $V_{CR}$.
 As a result, we obtain 
\begin{equation}
  \label{ndivCR}
  t_3 \, \partial_\Gamma \bO_{GW}(\CR) = 
\bO_{GW}(\Gamma_F)_0 \cdot \bO_{GW}(\CR) - \bO_{GW}(\CR)
\cdot  \bO_{GW}(\Gamma_F) \,,
\end{equation}
where $\bO_{GW}(\Gamma_F)_0$ denotes the
contribution of curves with degree 0 in the $\bA_n$-direction.
An identical equation holds in Donaldson-Thomas theory.

The horizontal part of the capped rubber operator is
simply equal to the tube,
$$\bO_{GW}(\CR)_0= \text{Id}, \ \ \bO_{DT}(\CR)_0= \text{Id}.$$
We will use \eqref{ndivCR} to uniquely reconstruct 
the capped rubber operators from their horizontal parts.

Let  $\bO_{GW}(\CR)_\sigma$ denote the contribution
of curves with
vertical degree $\sigma$. Equation \eqref{ndivCR}
may be written as
\begin{equation}
  \label{divisor2}
  t_3 \, (\Gamma\cdot\sigma) \, \bO_{GW}(\CR)_\sigma =
\big[\bO_{GW}(\Gamma_F)_0,  \bO_{GW}(\CR)_\sigma\big]  + \dots \,,
\end{equation}
where the dots stand for terms involving $\bO_{GW}(\CR)_{\sigma'}$
with $\sigma-\sigma'$ nonzero and effective. The latter terms may be
assumed known by induction.

As long as $\sigma\ne 0$, the left side of \eqref{divisor2}
can be made nonzero by a suitable choice of $\Gamma$.
Since the operator $\bO_{GW}(\Gamma_F)_0$ does not depend on
$t_3$, neither do the eigenvalues of the
commutation action of $\bO_{GW}(\Gamma_F)_0$.
Hence,
the linear equation \eqref{divisor2} has a unique
solution for $\bO_{GW}(\CR)_\sigma$ in the
the field of rational functions of $q$ and $t_i$,

An identical discussion is valid in Donaldson-Thomas theory.
In fact, since
$$\bO_{GW}(\Gamma_F)=\bO_{DT}(\Gamma_F)$$
by Theorem 2,
 the reconstruction is
the same in the two theories.
The GW/DT correspondence is therefore proven. 
\end{proof}

\subsection{Correspondence for 2-leg vertices}\label{stwoleg}
\label{pr1}

We first prove Proposition 2 for the capped vertices of
the form $\bC^*(\lambda,\nu,\emptyset)$. 
Since the last
partition is trivial, such vertices are said to
have {\em 2-legs}. Our constructions will be parallel
for Gromov-Witten and Donalson-Thomas theory, so we
omit the subscript in the capped vertex notation.

\begin{Lemma}\label{twoleg}
2-leg capped vertices satisfy the GW/DT correspondence.
\end{Lemma}

\begin{proof}
The cap geometery
$$
\pi:\bA_1 \times \Pp \rightarrow \Pp
$$
with relative conditions along $F_\infty$
has already been proven to satisfy the GW/DT correspondence
in Section \ref{scap}.

On the other hand,  the $\bA_1$-cap may be computed by relative capped
localization. The capped vertices and
edges over $0\in \Pp$ occur exactly as explained in Section 2.
Over $\infty \in \Pp$, a single vertex occurs 
given by capped $\bA_1$-rubber. Again, relative capped
localization is easily seen to be equivalent to usual
relative localization. In the starred normalization,
the relative capped localization
formulas for Gromov-Witten and Donaldson-Thomas theory
are exactly parallel --- the former  is obtained
from the latter by replacing all occurances of  $DT$ 
by $GW$.

\psset{unit=0.3 cm}
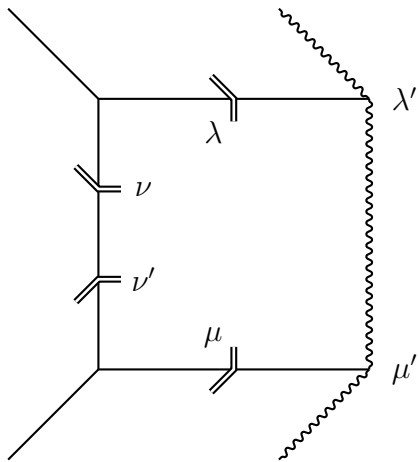
\begin{figure}[!htbp]
  \centering
   \begin{pspicture}(0,0)(19,20)
   \psline(4,4)(16,4)
   \psline(4,16)(16,16)
   \psline(0,0)(4,4)(4,16)(0,20)
   \pszigzag[coilarm=0.1,coilwidth=0.5,linearc=0.1](12,20)(16,16)
   \pszigzag[coilarm=0.1,coilwidth=0.5,linearc=0.1](16,4)(16,16)
   \pszigzag[coilarm=0.1,coilwidth=0.5,linearc=0.1](12,0)(16,4)
   \psline[doubleline=true](10,5)(10,4)(9,3)
   \psline[doubleline=true](10,15)(10,16)(9,17)
   \psline[doubleline=true](3,7)(4,8)(5,8)
   \psline[doubleline=true](3,13)(4,12)(5,12)
   \rput[tr](9.5,15){$\lambda$}
   \rput[l](17,16){$\lambda'$}
   \rput[l](17,4){$\mu'$}
   \rput[br](9.5,5){$\mu$}
   \rput[c](6,8){$\nu'$}
   \rput[c](6,12){$\nu$}
   \end{pspicture}
 \caption{Capped localization for the $\bA_1$-cap}
  \label{A1cap}
\end{figure}

A schematic
representation of the localization procedure
 is depicted in Figure \ref{A1cap}.
The lines in Figure \ref{A1cap} represent the edges of
the toric polyhedron of $X$. The squiggly lines belong to the
relative divisor. Partitions $\lambda'$ and $\mu'$ represent relative
conditions in the fixed point basis of $H^*_T(\bA_1,\mathbb{Z})$.
Other partitions represent intermediate relative conditions to be
 summed over.

Horizontal edges in  the $\Pp$-direction have a trivial normal
bundle. Hence, the corresponding edge operators are
identity operators --- which is why there is only one intermediate
partition on those edges. The other compact edge carries
two partitions satisfying
 $$|\nu|=|\nu'|$$
 connected by an capped $(0,-2)$-edge.
Figure \ref{A1cap} also 
represents two 2-legged capped vertices
$\bC^*(\lambda,\nu,\emptyset)$ and $\bC^*(\nu',\mu,\emptyset)$,
capped $\bA_1$-rubber 
connecting $(\lambda,\mu)$ to $(\lambda',\mu')$,
and 4 gluing factors indicated by double lines.

By Lemma \ref{Pcapr},  the 
capped $\bA_1$-rubber invariants  are rational functions of $q$
and satisfy the GW/DT correspondence. Since
$$
\bO_{GW}(\CR) =\bO_{DT}(\CR)= \text{Id} + o\left(s^E\right)\,,
$$
where $E$ is the exceptional divisor, the operator $\CR$ is
invertible.

We conclude the combination of
two 2-legged capped vertices and one capped edge
illustrated in Figure \ref{A1half} satisfies the
GW/DT correspondence for any value of $|\nu|=|\nu'|$\,.
\psset{unit=0.3 cm}
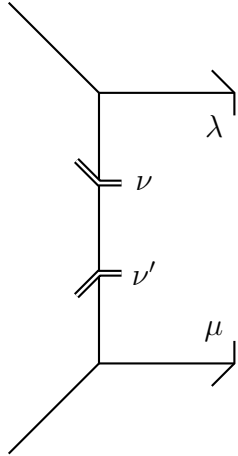
\begin{figure}[!htbp]
  \centering
   \begin{pspicture}(0,0)(10,20)
   \psline(4,4)(10,4)
   \psline(4,16)(10,16)
   \psline(0,0)(4,4)(4,16)(0,20)
   \psline(10,5)(10,4)(9,3)
   \psline(10,15)(10,16)(9,17)
   \psline[doubleline=true](3,7)(4,8)(5,8)
   \psline[doubleline=true](3,13)(4,12)(5,12)
   \rput[tr](9.5,15){$\lambda$}
   \rput[br](9.5,5){$\mu$}
   \rput[c](6,8){$\nu'$}
   \rput[c](6,12){$\nu$}
   \end{pspicture}
 \caption{$\bA_1$-cap minus the capped rubber}
  \label{A1half}
\end{figure}

We will now prove the GW/DT correspondence for 2-legged
capped
vertices $\bC^*(\lambda,\nu,\emptyset)$ by induction
on
$$
n=\min(|\lambda|,|\nu|) \,.
$$
The base case $n=0$ of the induction is provided
by the correspondence for 1-legged
vertices. The 1-legged vertex  is simply the $\bA_0$-cap
geometry treated in Section \ref{scap}.

We assume the partitions in Figure \ref{A1half} satisfy
$$
|\lambda| \ge |\nu| = |\nu'| > |\mu| \,.
$$
Then all capped vertices $\bC^*(\nu',\mu,\emptyset)$ are
known by induction. Figure \ref{A1half} may be
interpreted as a system of linear equations for
the unknown vertices $\bC^*(\lambda,\nu,\emptyset)$
in terms of the known vertices $\bC^*(\nu',\mu,\emptyset)$.
The linear system has more equations than unknowns.
Indeed, for fixed $|\lambda|$, the number of unknowns
equals $p(n)$, the number of partitions
of $n$, while the number of equations equals the number
of possibilities for $\mu$, 
$$
p(n-1)+\dots+p(1)+p(0) \ge p(n) \,, \quad n\ge 1 \,.
$$
Since the $(a,b)$-edge operator is invertible{\footnote{The
invertibility of the $(a,b)$-edge operator is easily
proven by consider the degeneration of a $(0,0)$-edge to
an $(a,b)$-edge and a $(-a,-b)$-edge.}},
the unique
solubility of the linear system is guaranteed by the
Lemma \ref{fundlem} below.

\end{proof}

\begin{Lemma}\label{fundlem}
The matrix of capped vertices
$$
\Big[\bC^*(\lambda,\mu,\emptyset)\Big],\quad |\lambda|=n\,,
|\mu|<n\,,
$$
has maximal rank.
\end{Lemma}

\begin{proof}

It is enough to consider the special \emph{topological
vertex} case
$$
t_1 + t_2 + t_3 =0\,,
$$
in which case great simplifications occur. The
capped DT vertex is related to the standard uncapped
DT vertex by invertible capped rubbers.
The uncapped DT vertex
may be evaluated directly, see
\cite{mnop1,ORV}.
Up to further invertible factors, the matrix to consider
becomes
$$
\sum_{\eta} s_{\lambda^{t}/\eta}(q^\rho)\, s_{\mu/\eta}(q^{\rho})
$$
where $s_{\lambda/\eta}$ are skew Schur functions evaluated
at
$$
q^\rho = (q^{-1/2}, q^{-3/2}, \dots) \,.
$$

As $\mu$ and $\eta$ vary over all partitions of 
size at most $n-1$, the matrix of skew Schur functions
$s_{\mu/\eta}$ is invertible and upper-triangular.  
We are thus reduced to proving the matrix
$$
\Big[s_{\lambda/\eta}(q^\rho)\Big]\,,
\quad |\lambda|=n\,, |\eta| < n\,,
$$
has maximal rank.

For a symmetric function $f$ of degree $r$, 
let 
$$f^{\perp}: \Lambda_k \rightarrow \Lambda_{k-r}$$
 denote the linear map adjoint to multiplication by $f$ under 
the standard inner product.  By a basic property of skew Schur functions,
$$s_{\rho}^{\perp}s_{\lambda} = s_{\lambda/\rho}$$
and also
$$p_{k}^{\perp}g = k \frac{\partial}{\partial p_k}\, g,$$
where $p_k$ is the power-sum symmetric function $\sum x_{i}^{k}$ and the derivative
is obtained by expressing $g$ as a polynomial expression of the functions $p_k$.

Using these two facts, the full rank statement 
on skew Schur functions is equivalent to the following claim.  
For any symmetric function $g$ of degree $n$,
there exists a partition $\mu$, $|\mu|<n$,
for which
\begin{equation}\label{gq2}
\left.\frac{\partial}{\partial p_{\mu_{1}}}\cdots  \frac{\partial}{\partial p_{\mu_{l}}}\, g\right|_{q^\rho} \ne 0 \,.
\end{equation}
Indeed, we can arrange the partial derivative \eqref{gq2}
 to be a multiple
of a single $p_k$ by
focusing on the leading lexicographic term of $g$ with respect to the
ordering
$$
p_1 > p_2 > p_3 > \dots \,.
$$
The proof of Lemma \ref{fundlem} and thus Lemma \ref{twoleg} is complete.
\end{proof}

\subsection{Proof of Proposition 2}
\label{pr2}

We now prove the GW/DT  correspondence for 3-legged
capped vertices $\mathsf{C}^*(\lambda,\mu,\nu)$. The argument follows 
the proof for the
2-legged vertices.
For the 3-leg case, we use capped localization for the 
$\bA_2$-cap. 
A schematic view of the localization is
illustrated in Figure \ref{A2cap}.
\psset{unit=0.3 cm}
\begin{figure}[!htbp]
  \centering
   \begin{pspicture}(0,0)(19,20)
   \psline(4,4)(16,4)
   \psline(4,16)(16,16)
   \psline(5,10)(17,10)
   \psline(0,0)(4,4)(5,10)(4,16)(0,20)
   \pszigzag[coilarm=0.1,coilwidth=0.5,linearc=0.1](12,20)(16,16)
   \pszigzag[coilarm=0.1,coilwidth=0.5,linearc=0.1](16,4)(17,10)
   \pszigzag[coilarm=0.1,coilwidth=0.5,linearc=0.1](17,10)(16,16)
   \pszigzag[coilarm=0.1,coilwidth=0.5,linearc=0.1](12,0)(16,4)
  \rput[l](17,16){$\lambda'$}
   \rput[l](18,10){$\eta'$}
   \rput[l](17,4){$\mu'$}
   \end{pspicture}
 \caption{Capped localization for the $\bA_2$-cap}
  \label{A2cap}
\end{figure}
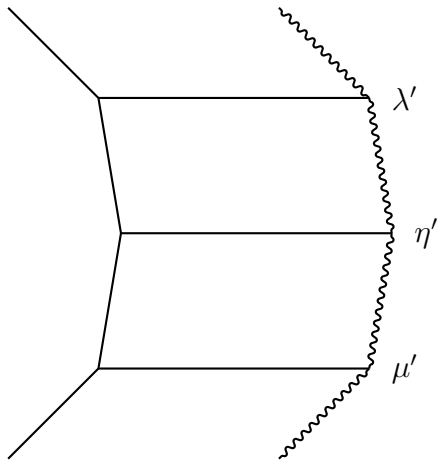

Partitions $\lambda'$, $\eta'$, and $\mu'$ represent 
arbitrary relative
conditions. Since the 3-leg vertex
occurs in Figure \ref{A2cap} once, the capped localization
may be viewed as a system of linear equations for the 3-leg
vertex. By Proposition \ref{fundlem}, the 3-leg
 vertex is uniquely determined.

In fact, the nondegeneracy of the 
square matrix
$$
\Big[\bC^*(\lambda,\mu,\emptyset)\Big],\quad |\lambda|,|\mu|\le n\,,
$$
is sufficient to conclude.
\qed

Since the GW/DT correspondence for toric varieties is
a consequence of Proposition 1 and 2, the proof of
Theorem 1 is complete.

\section{Properties of the vertex}\label{s_Proper}

\subsection{Connections}
\subsubsection{Topological vertex}
Let $X$ be a local Calabi-Yau toric 3-fold with a 
2-dimensional subtorus
$$T_{CY} \subset T$$
 preserving the
Calabi-Yau form.  A procedure for computing the
Gromov-Witten series via the topological vertex
was proposed
by Aganagic, Klemm, Mari\~no, and Vafa \cite{AKMV}. In \cite{mnop1,ORV},
the procedure of \cite{AKMV} was shown to 
exactly compute the Donaldson-Thomas series of $X$.
By the GW/DT correspondence of Theorem 1, the two theories
coincide.

The 1-leg version of the topological vertex
appeared first in
\cite{MV} and was proven in \cite{llz,OPMV}. The
2-leg case was proven in \cite{llz2}. 
Finally, an evaluation of the Gromov-Witten theory
of a local Calabi-Yau toric 3-fold $X$ in a form closely
resembling 
the topological vertex was obtained in \cite{lllz}.

\subsubsection{Stable pairs}
The virtual enumeration of stable pairs in the
derived category of $X$ was conjectured in \cite{PT1,PT2}
to be equivalent to both Gromov-Witten and Donaldson-Thomas
theory. 
After the triangle of equivalences of Section \ref{angeo}
is extended to include the stable pairs theory of
$\bA_n \times \Pp$,  our proof of
Theorem 1 will also apply to the stable pairs theory
of nonsingular toric 3-folds.

\subsection{Conjectures}
\subsubsection{Chow varieties}
Let $X$ be a nonsingular projective 3-fold with very ample
line bundle $L$.
The Chow variety of curves
parameterizes cycles of class $\beta\in H_2(X,\mathbb{Z})$ in $X$. 
More precisely, we define  $\Chow(X,\beta)$  to be  
the seminormalization of the subvariety of Chow forms associated to the
embedding $X \subset \mathbf{P}(H^0(X,L)^\vee)$.
The variety $\Chow(X,\beta)$ is independent of $L$.
See \cite{jk} for a detailed treatment.

For both the moduli space of stable maps
$\oM'_{g}(X,\beta)$ and the Hilbert scheme of curves
$I_n(X,\beta)$, the associated seminormalized varieties
admit maps to $\Chow(X,\beta)$
for all $g$ and $n$,
$$
\xymatrix{
\oM'_{g}(X,\beta)_{\mathrm{sn}} \ar[dr]^{\rho_{GW}} && I_n(X,\beta)_{\mathrm{sn}}
\ar[dl]_{\rho_{DT}}\\
  & \Chow(X,\beta)}\,.
$$
As a corollary of our proof of Theorem 1, we immediately
obtain the following result.

\begin{Corollary}\label{TChow} 
Let $X$ be a nonsingular projective toric 3-fold.
We have
\begin{multline}
  \sum_g u^{2g-2} \,
  {\rho_{GW}}_*\left(\left[\oM_{g,n}'(X,\beta)\right]^{\vir} \right) = \\
  \frac1{\bZ_{DT}(X,q)_0} \, \sum_n q^{n} \,
  {\rho_{DT}}_*\left(\left[I_n(X,\beta)\right]^{\vir} \right) \,
\end{multline}
in $H_*(\Chow(X,\beta),\mathbb{Q}) 
\otimes \mathbb{Q}(q)$, 
after the variable change $e^{iu}=-q$.
\end{Corollary}

\begin{proof}
Define the Chow variety with 
respect to a $T$-equivariant very ample bundle $L$ on $X$.
Then, both $\rho_{GW}$ and $\rho_{DT}$ are
also $T$-equivariant, and the respective push-forwards 
can be calculated by $T$-equivariant localization.
The Chow classes of $T$-fixed points 
in both Gromov-Witten and Donaldson-Thomas theory
are determined by the edge degrees. By the GW/DT correspondence
for capped vertices and edges, the contributions in both theories
are matched when the edge degrees are fixed.
\end{proof}

Since Chow and homology groups are preserved by seminormalization, 
the above homological push-forwards $\rho_{*}$ are well-defined.  
As the primary Gromov-Witten
 and Donaldson-Thomas
 invariants can be computed in terms of the push-forwards
$\rho_*$, Corollary \ref{TChow} may
be viewed as a refinement of Theorem \ref{aaa}.

\begin{Conjecture}
The equivalence of Corollary \ref{TChow} holds for all
nonsingular projective 3-folds $X$.
\end{Conjecture}

Unfortunately, 
the Chow statement of Corollary \ref{TChow} does not capture the full
GW/DT correspondence as descendent insertions are not pulled-back
from the Chow variety.

\subsubsection{Polynomiality}

We have
proven the capped vertex is a rational function
$$\mathsf{C}_{DT}(\lambda,\mu,\nu,t_1,t_2,t_3)\in \mathbb{Q}(q,t_1,t_2,t_3).$$
Using a combination of geometry and box counting, we can further
prove the following result.{\footnote{The proof together with
further properties of the capped vertex will be presented
in a future paper.}}

\begin{Theorem} We have
$$\mathsf{C}_{DT}(\lambda,\mu,\nu,t_1,t_2,t_3)\in \mathbb{Q}(q) \otimes
\mathbb{Q}(t_1,t_2,t_3)$$
with  possible 
poles in $q$ occuring {only at roots of unity} and 0.
\end{Theorem}

An identical statement is obtained for the capped edge by studying
the differential equations of the rubber calculus. 
An analysis of the singularities of the differential
equations proves a property encountered earlier by J. Bryan:
the capped $(a,b)$-edge
$\mathsf{E}_{DT}(\lambda,\mu,t_1,t_2,t_3,t_1',t_2')$
is a Laurent polynomial in $q$ if $a,b\geq 0$.
We expect polynomiality to also be true for the capped vertex.

\begin{Conjecture}
$\mathsf{C}_{DT}(\lambda,\mu,\nu,t_1,t_2,t_3) \in \mathbb{Q}(t_1,t_2,t_3) 
[q, \frac{1}{q}].$
\end{Conjecture}

\subsection{Calculations}

We denote by $\mathsf{C}^{\circ}_{DT}(\lambda,\mu,\nu,t_1,t_2,t_3)$ the
connected
version of the capped vertex. Omitting the weights
$t_i$ from the notation,
\begin{gather*}
\sum_{\lambda,\mu,\nu}\bC_{DT}(\lambda,\mu,\nu) x^\lambda y^\mu z^\nu=
\exp(\sum_{\lambda,\mu,\nu}\bC^{\circ}_{DT}(\lambda,\mu,\nu) x^\lambda
y^\mu z^\nu),\\
x^\lambda=\prod_i x_{\lambda_i},\quad y^\mu=\prod_i
y_{\mu_i},\quad z^\nu=\prod_i z_{\nu_i}.
\end{gather*}
Let $P_{\lambda,\mu,\nu}$ be the polynomial
\begin{equation*}
t_1^{\ell(\mu)+\ell(\nu)}t_2^{\ell(\lambda)+\ell(\nu)}t_3^{\ell(\lambda)+\ell(\mu)}
 \prod_{i=1}^{|\mu|}\prod_{j=1}^{|\lambda|}(it_1+jt_2
)\prod_{j=1}^{|\nu|}\prod_{k=1}^{|\mu|}(jt_2+kt_3
)\prod_{i=1}^{|\nu|}\prod_{k=1}^{|\lambda|}(it_1+kt_3 ).
\end{equation*}
Define $\bR(\lambda,\mu,\nu)$
by the formula
\begin{gather*}
\bC^{\circ}_{DT}(\lambda,\mu,\nu)=\bR(\lambda,\mu,\nu)
q^{1-|\lambda|-|\mu|+|\nu|}(1+q)^{|\lambda|+|\mu|+|\nu|-2}
\Pi_\lambda\Pi_\mu\Pi_\nu/P_{\lambda,\mu,\nu},\\
\Pi_\lambda=\prod_i(1-(-q)^{\lambda_i})/\mathfrak{z}(\lambda).
\end{gather*}
Our calculations suggest $\bR(\lambda,\mu,\nu)$ 
is a Laurent polynomial of $q$.

The
formula for $\bR$ is known in the 1-leg case,
$$ \bR(\lambda,\emptyset,\emptyset)=1,\mbox{ if }|\lambda|\leq 1,\mbox{
and } \bR(\lambda,\emptyset,\emptyset)=0 \mbox{ otherwise.}$$
We can also prove 
$$ \bR(\lambda,[1],\emptyset)=t_3^{\ell(\lambda)},\mbox{ if }
|\lambda|>0.$$
Below we give some further values of $\bR(\lambda,\mu,\nu)$
with small partitions.

\vspace{15pt}

\begin{tabular}{l|l}
$\lambda,\mu,\nu$& $\bR(\lambda,\mu,\nu)$\\
\hline &\\
 $1^2,1^2,\emptyset$ &$((t_1+t_2-t_3)(q+{q}^{-1})+(-10t_2-10t_1-2t_3))t_3^3$\\
 $1^2,2,\emptyset$ &   $((t_1+t_2-t_3)(q+{q}^{-1})
+(-6t_1-2t_3-8t_2))t_3^2$\\
 $2,2,\emptyset$ & $((t_1+t_2-t_3)(q+q^{-1})+(-2t_3-4t_2-4t_1))t_3$\\
 $1,1,1$&$(t_1+t_2)(t_2+t_3)(t_1+t_3)$\\
 $2,1,1$&$(t_1+t_2+t_3)(t_1+2t_2)(t_2+t_3)(t_1+2t_3)$\\
 $1^2,1,1$&$(t_3^2+t_1t_3+t_2^2+t_2t_3+t_1t_2)(t_1+2t_2)(t_2+t_3)(t_1+2t_3)$\\

\end{tabular}

\vspace{+8 pt}
\noindent
Department of Mathematics\\
Columbia University\\
dmaulik@math.columbia.edu

\vspace{+8 pt}
\noindent
Department of Mathematics\\
Princeton University\\
oblomkov@math.princeton.edu

\vspace{+8 pt}
\noindent
Department of Mathematics\\
Princeton University\\
okounkov@math.princeton.edu

\vspace{+8 pt}
\noindent
Department of Mathematics\\
Princeton University\\
rahulp@math.princeton.edu

\end{document}